\documentclass[reqno,twoside,10pt,final]{amsart} 

\usepackage{amsmath,amsthm,amssymb,amsfonts,mathrsfs,color,hyperref,xcolor}
\usepackage{bbm}
\usepackage{amsbsy}
\usepackage{latexsym}

\theoremstyle{plain}
\begingroup
\newtheorem{thm}{Theorem}[section]
\newtheorem{lem}[thm]{Lemma}
\newtheorem{prop}[thm]{Proposition}

\endgroup

\theoremstyle{definition}
\begingroup

\newtheorem{rem}[thm]{Remark}
\endgroup

\numberwithin{equation}{section}

\newcommand{\ds}{\displaystyle}

\newcommand{\N}{\mathbb N} 
\newcommand{\R}{\mathbb R} 
\newcommand{\Sn}{\mathbb S} 
 
\newcommand{\Ms}{{\mathbb M}^{n{\times}n}_{\rm sym}}
\newcommand{\Msd}{{\mathbb M}^{n{\times}n}_D}
\newcommand{\Mst}{{\mathbb M}^{3{\times}3}_{\rm sym}}

\newcommand{\E}{{\mathcal E}}

\newcommand{\wto}{\rightharpoonup}
\newcommand{\e}{\varepsilon}

\newcommand{\LL}{{\mathcal L}}
\newcommand{\HH}{{\mathcal H}}
\newcommand{\M}{{\mathcal M}}
\newcommand{\C}{{\mathcal C}}

\DeclareMathOperator{\tr}{tr}
\DeclareMathOperator{\Id}{Id}
\DeclareMathOperator{\dive}{div}
\DeclareMathOperator{\cof}{cof}

\let\O=\Omega

\newcommand{\restrict}{\begin{picture}(10,8)\put(2,0){\line(0,1){7}}\put(1.8,0){\line(1,0){7}}\end{picture}}

\begin{document}
 
\title[Concentration versus oscillation effects in brittle damage]{Concentration versus oscillation effects in\\brittle damage}

\author[J.-F. Babadjian]{Jean-Fran\c cois Babadjian}
\address[J.-F. Babadjian]{Laboratoire de Math\'ematiques d'Orsay, Univ. Paris-Sud, CNRS, Universit\'e Paris-Saclay, 91405 Orsay, France.}
\email{jean-francois.babadjian@math.u-psud.fr}

\author[F. Iurlano]{Flaviana Iurlano}
\address[F. Iurlano]{Sorbonne Universit\'e, CNRS, Universit\'e de Paris, Laboratoire Jacques-Louis Lions, F-75005 Paris, France}
\email{iurlano@ljll.math.upmc.fr}

\author[F. Rindler]{Filip Rindler}
\address[F. Rindler]{Mathematics Institute, University of Warwick, Coventry CV4 7AL, UK, and The Alan Turing Institute, British Library, 96 Euston Road, London NW1 2DB, UK}
\email{F.Rindler@warwick.ac.uk}

\date{\today}
\subjclass[2010]{}

\keywords{}

\begin{abstract}
This work is concerned with an asymptotic analysis, in the sense of $\Gamma$-convergence, of a sequence of variational models of brittle damage in the context of linearized elasticity. The study is performed as the damaged zone concentrates into a set of zero volume and, at the same time and to the same order $\varepsilon$, the stiffness of the damaged material becomes small. Three main features make the analysis highly nontrivial: at $\varepsilon$ fixed, minimizing sequences of each brittle damage model oscillate and develop microstructures; as $\varepsilon\to 0$, concentration and saturation of damage are favoured; and the competition of these phenomena translates into a degeneration of the growth of the elastic energy, which passes from being quadratic (at $\varepsilon$ fixed) to being of linear-growth type (in the limit). Consequently, homogenization effects interact with singularity formation in a nontrivial way, which requires new methods of analysis. In particular, the interaction of homogenization with singularity formation in the framework of linearized elasticity appears to not have been considered in the literature so far. We explicitly identify the $\Gamma$-limit in two and three dimensions for isotropic Hooke tensors. The expression of the limit effective energy turns out to be of Hencky plasticity type. We further consider the regime where the divergence remains square-integrable in the limit, which leads to a Tresca-type model.

\medskip

\noindent\textsc{Keywords} Brittle damage, variational model, asymptotic analysis, Hencky plasticity
\end{abstract}

\maketitle

\tableofcontents

\newpage

\section{Introduction}

In the theory of brittle damage (see, e.g.,~\cite{FM1}) in the so-called ``brutal'' regime, a linearly elastic material can exist in one of two states: a damaged state, for which the energy is described via a symmetric fourth-order ``weak'' elasticity (Hooke) tensor $\mathbf{A}_w$; or an undamaged state with a ``strong'' elasticity tensor $\mathbf{A}_s$, with $\mathbf{A}_w \leq \mathbf{A}_s$. Damage is a typical inelastic phenomenon described by means of an internal variable, which here is given as the characteristic function of the damaged region. The dissipational energy is taken as proportional to the damaged volume. If $\O \subset \R^n$ stands for the volume occupied by the body at rest, $u:\O \to \R^n$ ($n = 2$ or $n = 3$) is the displacement and $\chi:\O \to \{0,1\}$ is the characteristic function of the damaged region, then the total energy is given as
\[
(u,\chi) \mapsto E(u,\chi):= \frac12 \int_\O \bigl[\chi \mathbf{A}_w + (1-\chi)\mathbf{A}_s \bigr] e(u):e(u)\, dx + \kappa \int_\O \chi\, dx,
\]
where $\kappa > 0$ is the material toughness, i.e., the local cost of damaging a healthy part of the medium, and $e(u):=\frac12 (\nabla u + \nabla u^T)$ is the linearized strain. This type of energy functional is also encountered in the theory of shape optimization, where one aims to find an optimal shape (here $D:=\{\chi=1\}$) minimizing a cost functional (here the elastic energy) under a volume constraint. In this framework, the toughness $\kappa$ can be thought of as a Lagrange multiplier associated to this equality constraint.

Assuming standard symmetry and ellipticity conditions on the elasticity tensors $\mathbf{A}_w$ and $\mathbf{A}_s$, the above energy $E$ is well-defined for displacements $u \in H^1(\O;\R^n)$. It is well known that the problem of minimizing $E$ (adding suitable forces and/or boundary conditions) is ill-posed, in the sense that minimizing sequences tend to highly oscillate and develop microstructure (see, e.g.,~\cite{FM1,KS}). A relaxation phenomenon occurs, leading to a homogenized problem where brittle damage is replaced by progressive damage. In this new formulation, damage is described by means of a volume fraction $\theta \in L^\infty(\O;[0,1])$ and the homogenized stiffness of a composite material is obtained through fine mixtures between the damaged part with volume fraction $\theta$ and the undamaged part with volume fraction $1-\theta$. Much work has been devoted to the study of this relaxed problem in homogenization theory, for example to the identification of all attainable composite materials (the so-called $G$-closure set), or to bounds on the effective coefficients (the Hashin-Shtrikman bounds). We refer to~\cite{MT,T,FraMur,AK1,AK2,KS} and to the monograph~\cite{A} as well as the references therein for more details. 

Minimizing $E$ first with respect to $\chi$, the relaxation problem described above can be rephrased as the identification of the lower semicontinuous envelope of the functional
$$u \in H^1(\O;\R^n) \mapsto \int_\O W (e(u))\, dx,$$
where 
$$W(\xi):=\min \left\{ \frac12 \mathbf{A}_s\xi:\xi ,\frac{1}{2}\mathbf{A}_w\xi:\xi +\kappa \right\}.$$
Notice in particular that $W$ fails to be (quasi-)convex. Standard relaxation results show that the lower semicontinuous envelope is given by
$$u \mapsto \int_\O SQW (e(u))\, dx,$$
where $SQW$ is the symmetric quasiconvex envelope of $W$. An explicit expression for $SQW$ is in general unknown, although several results have been obtained, see, for instance,~\cite{AF,AL}.

In the present work, we are interested in the limit passage to a total damage model, i.e., when the elasticity coefficients $\mathbf{A}_w$ of the weak material tend to zero, and at the same time the volume of the damaged region vanishes. More precisely, we introduce a small parameter $\e > 0$ and consider the rescaled energy functional
\[
  E_\e(u,\chi):=\frac12 \int_\O \bigl[ \eta_\e \chi \mathbf{A}_w + (1-\chi)\mathbf{A}_s \bigr] e(u):e(u)\, dx + \frac{\kappa}{\e} \int_\O \chi\, dx,
\]
where $\eta_\e \to 0$ as $\e \to 0$ is a rescaling factor. We then ask about the limit behavior of $E_\e$ as $\e \to 0$. Note that now there is a trade-off between the cost of the damage $\kappa/\e$ and the resulting weakening of the stiffness tensor $\eta_\e \mathbf{A}_w$ in the damaged region. 

One motivation of this analysis goes back to the numerical investigations performed in~\cite{AJVG} in a discrete framework. There, forcing the elastic properties to become weaker and weaker on sets of arbitrarily small measure leads to the appearance of concentrations. A first aim of this paper is to make rigorous such observations and to precisely describe the limit model obtained through an asymptotic analysis.

From a mathematical point of view, we will carry out our analysis by computing the $\Gamma$-limit of $E_\e$ as $\e \to 0$ for the three possible regimes of $\eta_\e \ll \e$, $\eta_\e \sim \e$ and $\eta_\e \gg \e$. It turns out that the most relevant regime is $\eta_\e \sim \e$. Indeed, on the one hand, if $\eta_\e \ll \e$, the elastic energy associated with the damaged material is so negligible that we obtain a trivial $\Gamma$-limit (see Theorem~\ref{thm:alpha=0}); we here do not address the question whether a suitable rescaling of the energy gives rise to a non-vanishing limit. Indeed, according to the proof of Theorem~\ref{thm:alpha=0}, it turns out that the energy scales like $\sqrt{\eta_\e/\e}$ so that we expect the right energy rescaling to be $\sqrt{\e/\eta_\e} E_\e$. On the other hand, if $\eta_\e \gg \e$, the damaged set is so small that the limit model turns out to be of pure elasticity type with elasticity tensor $\mathbf{A}_s$ (see Theorem~\ref{thm:alpha=infty}). 

The case $\eta_\e \sim \e$ poses a number of mathematical challenges. First, as $\e \to 0$, it is not hard to see that, if $u_\e$ denotes an almost-infimum point of $E_\e$, the only uniform bound that can be obtained is on the $L^1$-norm of the elastic strains $(e(u_\e))_{\e>0}$  (see Lemma~\ref{lem:comp1}). This shows that $e(u_\e)$ may concentrate into a singular measure in the limit, which describes ``condensated'' defects inside the medium. The domain of the displacements in the $\Gamma$-limit is thus given by $BD(\O)$, the space of vector fields of bounded deformation (see the next section for a precise definition). Second, to compute the $\Gamma$-limit of $E_\e$, we need to take into account that homogenization effects  will interact with the formation of concentrations in a nontrivial way. We are not aware of any previous works considering the above framework. We remark that the quadratic-to-linear behavior arising from energetic competition is typical of works in the gradient theory of phase transition~\cite{FT,BF}, where, however, the full gradient is considered in place of the symmetric gradient; a quadratic-to-linear-type behavior in the context of linearized elasticity is obtained in~\cite{BDV1,BDV}, but there the relaxation concerns a functional defined on functions that are smooth outside the free-discontinuity set; finally, explicit identifications of the $\Gamma$-limit in linearized elasticity are available for quadratic-to-quadratic convergences~\cite{FI,CCF,C,CC}.
To conclude this bibliographic overview, let us mention an interesting connection with the optimal design problems studied in~\cite{O1} and, more recently, in~\cite{O2}. They appear as a dual version of our problem, being formulated in terms of the stress $\sigma:=\mathbf{A}e(u)$ instead of the strain $e(u)$. From a technical point of view, the main difference with our work is that only one phase (the weak phase) is considered there. This permits to prove the $\Gamma$-liminf inequality in a more abstract way through a careful change of the boundary datum.

The identification of the $\Gamma$-limit is highly nontrivial because of the inherent nonconvexity of the problem. Assuming for simplicity that $\eta_\e = \e$, the problem of finding the $\Gamma$-limit of $E_\e$ turns out to be equivalent to finding the $\Gamma$-limit of the family of functionals 
$$u \in H^1(\O;\R^n) \mapsto \int_\O W_\e(e(u))\, dx,$$
where
\[
W_\e(\xi):=\min \left\{ \frac12 \mathbf{A}_s\xi:\xi ,\frac{\e}{2}\mathbf{A}_w\xi:\xi +\frac{\kappa}{\e} \right\},
\]
or still the $\Gamma$-limit of their relaxations, given by
$$u \mapsto \int_\O SQW_\e(e(u))\, dx,$$
where $SQW_\e$ is the symmetric quasiconvex envelope of $W_\e$. We next specialize to isotropic Hooke tensors $\mathbf{A}_w$ and $\mathbf{A}_s$, that is,
\begin{align*}
\mathbf{A}_w \xi & := \lambda_w (\tr\xi)\, \Id+ 2\mu_w \xi,\\
\mathbf{A}_s \xi & := \lambda_s (\tr\xi)\, \Id+ 2\mu_s \xi,
\end{align*}
where $\lambda_i>0$ and $\mu_i>0$ are the Lam\'e coefficients.
In this case, although the explicit expression of $SQW_\e$ is not known (see~\cite{AL}), it is possible to compute explicitly its pointwise limit $\overline W$, which rests on an interesting $\Gamma$-convergence argument for the Hashin-Shtrikman bound (see Proposition~\ref{prop:Walpha}). More precisely, the pointwise limit $\overline W$ is given as an infimal convolution
$$
\overline W(\xi) := (f \, \Box \, \sqrt{2\kappa h})(\xi) :=\inf_{\xi' \in \Ms} \bigl\{ f(\xi-\xi')+\sqrt{2\kappa h(\xi')} \bigr\},  \qquad \xi\in\Ms,
$$
where 
$$f(\xi) := \frac12 \mathbf{A}_s\xi:\xi \quad\text{and}\quad h(\xi) :=\mu_w \left(\sum_{i=1}^n|\xi_i| \right)^2+(\lambda_w+\mu_w) \left(\sum_{i=1}^n \xi_i \right)^2,$$
with the $\xi_i$'s denoting the eigenvalues of $\xi$.

Our main result (see Theorem~\ref{thm:alpha=1}) is then that the functionals $E_\e$ $\Gamma$-converge as $\e\to 0$ to the functional
\[
\label{energy_intro}
u \in BD(\O) \mapsto \int_\O \overline W(e(u))\, dx + \int_\O \overline W^\infty \left( \frac{dE^s u}{d|E^s u|} \right)\, d|E^s u|,
\]
where $\overline W^\infty$ is the recession function of $\overline W$ and the linearized strain measure $Eu$ is decomposed (in the Lebesgue--Radon--Nikod\'ym sense) as $Eu = e(u) \mathcal{L}^n + E^s u$. The function $\overline{W}$ turns out to be quadratic close to the origin and to grow linearly at infinity, with a slope given by the recession function $\overline W^\infty=\sqrt{2\kappa h}$. Remarkably, and perhaps surprisingly, this is a typical energy density encountered in perfect plasticity (actually, Hencky plasticity, since we are dealing with static models). So, our results show how a brittle damage model may lead to a plasticity model in a singular limit (see also~\cite{I,DMI} for gradient damage models).

This result entails that for the bulk part we have a response that is (optimally) homogenized between the undamaged and the damaged parts, while for the singular part (which may contain jumps and fractals) we only see a dependence on the damaged Hooke tensor $\mathbf{A}_w$. Since for $\xi \in \Ms$ the expression $\sqrt{2\kappa h(\xi)}$ describes the energy cost (density) of optimally damaging the linear map $x \mapsto \xi x$, the above expression for the $\Gamma$-limit can be interpreted as follows: in the bulk part, the material may oscillate finely between damaged and undamaged areas, giving, by definition of the infimal convolution, a decomposition of the homogenized bulk energy of the form
$$\overline W(\xi)=\frac12 \mathbf{A}_se:e +\sqrt{2\kappa  h(p)},$$
where the linearized strain is additively split as $\xi=e+p$ with $e$ an elastic strain and $p$ a plastic (permanent) strain.

For the proof of the theorem, one first observes that the effective integrand $\overline W$ is a natural candidate for the bulk energy density of the $\Gamma$-limit and the energy functional associated to it easily provides an upper bound for $E_\e$. We stress that it is not straightforward to obtain the $\Gamma$-limsup inequality through a direct construction of a recovery sequence. Explicit constructions can, however, be exhibited if the displacement is linear $u(x)=\xi x$ and the matrix $\xi$ is diagonal, and improved if $\xi$ is rank-one symmetric (see Section~\ref{Hencky}).

The problem of establishing the lower bound is much more delicate. The crucial question  is to understand the interplay between the shape of $SQW_\e$ and a sequence of symmetric gradients. These questions are in general highly nontrivial and not much is known. The only results about concentrations in sequences in $BD(\O)$ seem to be~\cite{DPR,DPR2}; also see the recent survey~\cite{DPR3}. The main difficulty is related to the fact that there is a loss in the growth of the elastic energy passing to the limit as $\e \to 0$, which prevents one to easily control the contribution of the energy for large strains. In addition, contrary to~\cite{O2,O1}, standard cut-off techniques, which replace the boundary value of a minimizing sequence by that of the target, do not apply since minimizing sequences only converge in the weak* sense in $BD$ (thus strongly in $L^p$ for any $p<\frac{n}{n-1} \leq 2$ by compact embedding), while the energy has quadratic growth for fixed $\e$. 

The classical argument to get a lower bound is to apply Young's inequality inside the damaged region. This allows us to bound from below the energy associated to arbitrary sequences $(\chi_\e)_{\e>0}$ and $(u_\e)_{\e>0}$ by
\begin{equation}\label{1617}
\int_\O (1-\chi_\e) \frac12 \mathbf{A}_s e(u_\e):e(u_\e) + \chi_\e \sqrt{2\kappa \mathbf{A}_w e(u_\e):e(u_\e)}\, dx.
\end{equation}
One observes
$$\sqrt{2\kappa \mathbf{A}_w \xi:\xi}\leq\sqrt{2\kappa h(\xi)}$$ and that equality holds only on rank-one symmetric matrices $a\odot b$ (see Proposition~\ref{prop:W_alpha}). Hence, this lower bound would coincide with the previous upper bound if $e(u_\e)(x)$ was rank-one symmetric for almost every $x\in\{\chi_\e=1\}$, which, however, is obviously false. 

Analyzing for simplicity the two-dimensional case, one observes that, when $e(u_\e)$ is not rank-one symmetric, the gap originating from replacing $\mathbf{A}_we(u_\e):e(u_\e)$ by $h(e(u_\e))$ in \eqref{1617} is controlled by the quantity $(\det(e(u_\e)))^+$. Now, heuristically, since $|e(u_\e)\chi_\e|\sim 1/\e$, one imagines that the subset, say $Z_\e$, where $u_\e$ has slope $1/\e$ along two different directions (in the sense that $e(u_\e)$ fails to be rank-one symmetric and has both eigenvalues of order $1/\e$) has measure of order strictly smaller than $\e$. If one would be able to formalize this idea, the two bounds obtained from below and from above would match. This intuition is supported by the fact that $e(u_\e)$ on $Z_\e$ is away from the wave cone associated to the differential operator ${\rm curl}\, {\rm curl}$, so that by~\cite{DPR} it is reasonable to expect some elliptic regularity properties for $u_\e$ in $Z_\e$ and therefore a good size estimate for $Z_\e$. However, the formalization of this ``compensated compactness'' strategy is at present unclear and we here must follow a different argument (which can, in fact, itself also be seen as a ``compensated compactness'' approach).

The key observation enabling our proof is that $\sqrt\e u_\e\rightharpoonup 0$ weakly in $H^1(\O;\R^n)$ and therefore in dimension $n=2$ one has ${\e\rm det}(\nabla u_\e)\rightharpoonup 0$ weakly* in the sense of measures. Fine computations are needed to adapt this observation to the symmetric gradient, then to its positive part, and, finally, to generalize the argument to three dimensions, where the condition ${\e\rm det}(\nabla u_\e)\rightharpoonup 0$ has to be replaced by ${\e\rm cof}(\nabla u_\e)\rightharpoonup 0$ with $\cof(\xi)$ the cofactor matrix associated to $\xi$.

In the same spirit as the model described above, we also study the asymptotic behavior of a similar family of functionals, where now the divergence term of the weak material does not degenerate to zero. More precisely, we consider a weak material with an elasticity tensor $\mathbf{A}_w^\e$ of the form
$$\mathbf{A}_w^\e \xi:= \lambda_w (\tr\xi)\Id+2\e\mu_w \xi,$$
where $\lambda_w \leq \lambda_s$. For all $(u,\chi) \in H^1(\O;\R^n) \times L^\infty(\O;\{0,1\})$, the associated energy is defined by
$$\widetilde E_\e(u,\chi):= \frac12 \int_\O \bigl[ \chi \mathbf{A}^\e_w + (1-\chi)\mathbf{A}_s \bigr] e(u):e(u)\, dx + \frac{\kappa}{\e} \int_\O \chi\, dx.$$
In this new problem, the divergence of the displacement is not penalized anymore, and the domain of the $\Gamma$-limit is given by those displacements $u \in BD(\O)$ satisfying $\dive u \in L^2(\O)$ (that is, the distributional divergence is absolutely continuous with respect to Lebesgue measure and has a square summable density). In other words, this means that the displacement $u$ lies in the Temam--Strang space $U(\O)$, see, e.g.,~\cite{Temam}. Using the same type of arguments, we show that the $\Gamma$-limit is a quadratic functional of $\dive u$ and a linear-growth functional of the deviatoric part $E_D u$ of the linearized strain measure $Eu$. It is explicitly given by
$$u \mapsto \int_\O\left(\frac{\lambda_s}{2}+\frac{\mu_s}{n}\right)(\dive u)^2 \, dx
 + \int_\O \widetilde W(e_D(u))\, dx + \int_\O  \sqrt{2\kappa \tilde h\left(\frac{dE^s_D u}{d|E^s_D u|} \right)} \,d|E^s_D u|,$$
where the deviatoric bulk energy density is again defined via an infimal convolution, namely as
$$\widetilde W:=\tilde f \, \Box \, \sqrt{2\kappa \tilde h}$$
with
$$\tilde f(\xi):=\mu_s|\xi|^2, \qquad \tilde h(\xi):=\mu_w \left(\sum_{i=1}^n|\xi_i| \right)^2 \quad \text{ for all }\xi \in \Msd$$
and $\xi_1 \leq \cdots \leq \xi_n$ being the ordered eigenvalues of $\xi$. We recover in this way the well-known Tresca model of perfect plasticity since $\sqrt{2\kappa\tilde h}$ is precisely the support function of the Tresca elasticity set $\widetilde K:=\bigl\{\tau \in \Msd : \tau_n-\tau_1 \leq 2 \sqrt{2\kappa \mu_w}\bigr\}$, where again $\tau_1 \leq \cdots \leq \tau_n$ are the ordered eigenvalues of the deviatoric matrix $\tau \in \Msd$.

The analysis carried out in this work is only concerned with the understanding of effective limit energies at the static level. If the body is subjected to time-progressive loads or boundary conditions, it is natural to go further to a time dependent model in the framework of quasistatic evolution under an irreversibility constraint on the damage process (see~\cite{FM1}). However, the understanding of the interplay between relaxation and irreversibility is usually a delicate issue, see e.g.~\cite{FG} for an energy-based model and~\cite{GL} for a threshold-based model. At present, it is unknown how irreversibility for fixed $\e>0$ is translated into the limit evolution model, if both approaches give rise to the same limit model of perfect plasticity as $\e \to 0$, and if irreversibility can change the limit model with respect to the static problem. For a passage
to the limit in a formally similar problem including an irreversibility
condition, see the forthcoming paper~\cite{BCI}

\medskip

This paper is organized as follows. In Section 2, we introduce general notation and define precisely the problem under investigation. In Section 3, we analyze the main regime $\eta_\e \sim \e$, leading to a Hencky-type model. Sections 4 and 5 are devoted to investigating the trivial regime $\eta_\e \ll \e$ and the elastic regime $\eta_\e\ll \e$. Finally, in Section 6, we carry out the analysis of the modified problem leading to a Tresca-type model.

\section{Notation and preliminaries}

\subsection{Notation}

The Lebesgue measure in $\R^n$ is denoted by $\mathcal L^n$ and $\HH^k$ stands for the $k$-dimensional Hausdorff (outer) measure. If $a$ and $b \in \R^n$, we write $a \cdot b:=\sum_{i=1}^n a_i b_i$ for the Euclidean scalar product, and we denote the corresponding norm by $|a|:=\sqrt{a \cdot a}$. 

\medskip

{\it Matrices.} The space of symmetric $n \times n$ matrices is denoted by $\Ms$. It is endowed with the Frobenius scalar product $\xi:\eta:=\tr(\xi \eta)$ and with the corresponding Frobenius norm $|\xi|:=\sqrt{\xi:\xi}$. We also denote by $\Msd$ the set of all symmetric deviatoric matrices, i.e.\ all $\xi \in \Ms$ such that $\tr\xi=0$. Any matrix $\xi \in \Ms$ can be uniquely decomposed as $\xi=\xi_D+\frac{\tr\xi}{n} \Id$, where $\xi_D:=\xi-\frac{\tr\xi}{n} \Id \in \Msd$ is the deviatoric part of $\xi$, and $\frac{\tr\xi}{n} \Id$ is the hydrostatic part of $\xi$. Finally, given $\xi \in \Ms$, we denote by $\det \xi$ its determinant and by $\cof \xi \in \Ms$ its cofactor matrix.
For any $a$, $b \in \R^n$, we define the tensor product $a \otimes b:=ab^T$ and the symmetric tensor product $a \odot b:=(a\otimes b+b\otimes a)/2$.

We recall two lemmas from linear algebra:

\begin{lem}\label{lem:aodotb}
Let $a$ and $b \in \R^n$. Then, the matrix $a \odot b$ has at most rank $2$, and in this case the nonzero eigenvalues have opposite signs. Conversely, if $\xi \in \Ms$ has rank two and the two nonzero eigenvalues have opposite signs, then there are $a,b \in \R^n$ such that $\xi = a \odot b$.
\end{lem}

\begin{lem}\label{lem:polyconvex}
For all $\xi \in \mathbb M^{3 \times 3}_{\rm sym}$, the matrix $\cof(\xi)$ is diagonalizable in the same orthonormal basis as $\xi$. In addition, if $\xi_1$, $\xi_2$, and $\xi_3$ are the eigenvalues of $\xi$, then $\xi_2\xi_3$, $\xi_1 \xi_3$ and $\xi_1\xi_2$ are the eigenvalues of $\cof(\xi)$.
\end{lem}

A proof of the first lemma is in~\cite[Lemma~2.2]{DPR3} and the second lemma follows from the fact that commuting symmetric matrices share a basis of eigenvectors.

\medskip

{\it Function spaces.} We use standard notation for Lebesgue spaces, $L^p$, and Sobolev spaces, $W^{k,p}$ or $H^k:=W^{k,2}$. Given an open subset $\O$ of $\R^n$, we denote by $BD(\O)$ the space of functions of bounded deformation, i.e., all vector fields $u \in L^1(\O;\R^n)$ such that the distributional linearized strain $Eu:=(Du+Du^T)/2 \in \mathcal M(\O;\Ms)$, where $\mathcal M(\O;\Ms)$ stands for the space of all $\Ms$-valued Radon measures with finite total variation. 
We can split $Eu$ according to the Lebesgue decomposition as
\[
  Eu = e(u) \mathcal{L}^n \restrict \O + E^s u = e(u) \mathcal{L}^n \restrict \O + \frac{dE^s u}{d|E^s u|} |E^s u|,
\]
where $e(u) \in L^1(\O;\Ms)$ is the Radon--Nikod\'ym derivative of $Eu$ with respect to $\mathcal{L}^n$, and $E^s u$ is the singular part of $Eu$ with respect to $\mathcal{L}^n$. Furthermore, we denote by $\frac{dE^s u}{d|E^s u|}$ the Radon--Nikodym derivative of $E^s u$ by its own total variation measure $|E^s u|$, i.e.\ the polar of $E^s u$. We refer to~\cite{ST1,Suquet,Temam,ACDM,DPR3} for general properties of the space $BD(\O)$. We also define $LD(\O):=\{u \in BD(\O) : \; E^s u=0\}$.

\medskip

{\it Convex analysis.} We recall several definitions and basic facts from convex analysis (we refer to~\cite{ET,R} for proofs). Let $\psi:\Ms \to [0,+\infty]$ be a proper function (i.e.\ not identically $+\infty$). The convex conjugate of $\psi$ is defined as
\begin{equation}\label{convconj}
\psi^*(\tau):=\sup_{\xi \in \Ms} \bigl\{ \tau:\xi - \psi(\xi) \bigr\} \quad \text{ for all }\xi \in \Ms,
\end{equation}
which is a convex and lower semicontinuous function. Repeating the process, we can define the biconjugate function $\psi^{**}:=(\psi^*)^*$, which turns out to be the lower semicontinuous convex hull of $\psi$, i.e., the largest lower semicontinuous and convex function below $\psi$. In particular, if $C \subset \Ms$ is a set, we define the indicator function $I_C$ of $C$ as $I_C:=0$ in $C$ and $+\infty$ otherwise. The convex conjugate $I_C^*$ of $I_C$ is called the support function of $C$.

If $k: \Ms \to [0,+\infty]$ is a positively $1$-homogeneous convex function such that $k(0)=0$, the polar function of $k$ is defined by 
$$k^\circ(\xi):=\sup_{k(\tau) \leq 1} \tau:\xi \quad \text{ for all }\xi \in \Ms.$$

Let $\phi:\Ms \to [0,+\infty)$ be a convex function. Then, the limit
$$\phi^\infty(\xi):=\lim_{t \to +\infty}\frac{\phi(t\xi)}{t}$$
exists for every $\xi \in \Ms$ (in $[0,+\infty]$), and $\phi^\infty$ is called the recession function of $\phi$. It is a convex positively $1$-homogeneous function.

If $\phi_1,\phi_2 : \Ms \to [0,+\infty]$ are proper convex functions, then the infimal convolution of $\phi_1$ and $\phi_2$ is defined as
\begin{equation}\label{infconv}
(\phi_1 \, \Box \, \phi_2)(\xi):=\inf_{\xi' \in \Ms} \bigl\{ \phi_1(\xi-\xi')+\phi_2(\xi') \bigr\},
\end{equation}
which turns out to be a convex function. It can be shown that
\[
  \phi_1 \, \Box \, \phi_2=(\phi_1^*+\phi_2^*)^*.
\]
Moreover, if $\phi_1$ and $\phi_2$ are nonnegative, convex, $\phi_1(0)=0$, and $\phi_2$ is positively $1$-homogeneous, then $\phi_1 \, \Box \, \phi_2$ is the convex hull of $\phi_1\wedge\phi_2:=\min(\phi_1,\phi_2)$.

If $\psi,\phi_1,\phi_2$ are defined on $\Msd$ only, then the convex conjugate and the inf-convolution can be defined as functions on $\Msd$, taking respectively the supremum and the infimum in the formulas~\eqref{convconj} and~\eqref{infconv} over the space $\Msd$.

\subsection{Description of the problem}

Let $\O$ be a bounded open set of $\R^n$. For every $u \in H^1(\O;\R^n)$, $\chi \in L^\infty(\O;\{0,1\})$ and any $\e>0$, we define the following brittle damage energy functional:
$$E_\e(u,\chi):=\frac12 \int_\O \bigl[ \eta_\e \chi \mathbf{A}_w + (1-\chi)\mathbf{A}_s \bigr] e(u):e(u)\, dx + \frac{\kappa}{\e} \int_\O \chi\, dx.$$
In the previous expression, $\kappa>0$, $\eta_\e>0$, and $\mathbf{A}_w$, $\mathbf{A}_s$ are symmetric fourth-order tensors satisfying
\begin{equation}\label{eq:growth}
c_i\, \Id \leq \mathbf{A}_i \leq c'_i\, \Id \quad \text{ for } i\in \{w,s\}
\end{equation}
as quadratic forms over $\Ms$, for some constants $ c_w, c_s, c'_w,c'_s>0$. 

We assume that $\eta_\e \to 0$ as $\e \to 0$, so that one can suppose that $\eta_\e \mathbf{A}_w \leq \mathbf{A}_s$ as quadratic forms. The Hooke tensors $\eta_\e \mathbf{A}_w$ and $\mathbf{A}_s$ represent respectively the elasticity coefficients of a weak and a strong material. The weak, or damaged, part of the body has elastic properties which degenerate. At the same time, the toughness $\kappa/\e \to+\infty$ as $\e \to 0$ forces the damaged zones to concentrate on vanishingly small sets. Our goal is to understand the behavior of the previous brittle damage functional as $\e \to 0$ by means of a  $\Gamma$-convergence analysis.

Let us define for all $\xi \in \Ms$,
\begin{equation}
\label{d:fg}
f(\xi):=\frac12 \mathbf{A}_s\xi:\xi, \quad g_\e(\xi):=\frac{\eta_\e}{2}\mathbf{A}_w\xi:\xi +\frac{\kappa}{\e}
\end{equation}
and
\[
  W_\e(\xi) := f(\xi) \wedge g_\e(\xi) = \min\{f(\xi), g_\e(\xi)\}.
\]
Then, we can write
\[
  E_\e(u,\chi) = \frac12 \int_\O (1-\chi)f(e(u)) + \chi g_\e(e(u)) \, dx.
\]

For all $(u,\chi) \in L^1(\O;\R^n) \times L^1(\O)$, we further set
\[
\E_\e(u,\chi) :=
\begin{cases}
E_\e(u,\chi) & \text{ if } (u,\chi) \in H^1(\O;\R^n) \times L^\infty(\O;\{0,1\}),\\
+\infty & \text{ otherwise.}
\end{cases}
\]
Let us remark that, provided suitable (Dirichlet) boundary conditions are applied on some portion of the boundary and/or external body loads are incorporated into the model, the application of Poincar\'e and Korn type inequalities (see~\cite{Temam}) show that the condition $e(u) \in L^2(\O;\Ms)$ is equivalent to $\nabla u \in L^2(\O;\mathbb M^{n \times n})$.

We consider the $\Gamma$-lower and $\Gamma$-upper limits $\E_0'$ and $\E_0'':L^1(\O;\R^n) \times L^1(\O) \to [0,+\infty]$, respectively, of $(\E_\e)_{\e>0}$, that is (see~\cite{DM}),  for all $(u,\chi) \in L^1(\O;\R^n) \times L^1(\O)$,
$$\E_0'(u,\chi):=\inf\left\{\liminf_{\e \to 0}\E_\e(u_\e,\chi_\e) :\;  (u_\e,\chi_\e) \to (u,\chi) \text{ in }L^1(\O;\R^n) \times L^1(\O)\right\},$$
and
$$\E_0''(u,\chi):=\inf\left\{\limsup_{\e \to 0}\E_\e(u_\e,\chi_\e) :\;  (u_\e,\chi_\e) \to (u,\chi) \text{ in }L^1(\O;\R^n) \times L^1(\O)\right\}.$$
If $\E_0' = \E_0''$, then this functional is the $\Gamma$-limit of the sequence $(\E_\e)_{\e>0}$. It is our task in the following to explicitly identify this functional. It turns out that this depends on the sequence $(\eta_\e)_{\e>0}$ (only) through the value
\begin{equation} \label{eq:alpha}
\alpha :=\lim_{\e\to 0} \frac{\eta_\e}{\e} \in [0,+\infty].
\end{equation}
We consider the sequence $(\eta_\e)_{\e>0}$ fixed, so we do not make the dependence on $\alpha$ explicit in our notation.

We begin our analysis by identifying the domain of finiteness of the $\Gamma$-limit.

\begin{lem}\label{lem:comp1}
Let $(u,\chi) \in L^1(\O;\R^n) \times L^1(\O)$ be such that $\E_0'(u,\chi)<+\infty$. Then, $\chi=0$ a.e.\ in $\O$ and  if further $\alpha > 0$, then $u \in BD(\O)$.
\end{lem}

\begin{proof}
Let $(u_\e,\chi_\e)_{\e>0}$ be a sequence such that $(u_\e,\chi_\e) \to (u,\chi)$ in $L^1(\O;\R^n) \times L^1(\O)$ and, for some $\delta > 0$,
$$\liminf_{\e\to 0} \E_\e(u_\e,\chi_\e)\leq\E_0'(u,\chi)+\delta<+\infty.$$
Let us extract a subsequence $(u_k,\chi_k)_{k \in \N} := (u_{\e_k},\chi_{\e_k})_{k \in \N}$ of $(u_\e,\chi_\e)_{\e>0}$ such that
$$\lim_{k \to \infty} \E_{\e_k}(u_k,\chi_k)=\liminf_{\e \to 0}\E_\e(u_\e,\chi_\e)<+\infty.$$
This implies that, for $k$ large enough, $u_k \in H^1(\O;\R^n)$, $\chi_k \in L^\infty(\O;\{0,1\})$, and
\[ \label{eq:energy-bound}
M:=\sup_{k \in \N} E_{\e_k}(u_k,\chi_k)<+\infty.
\]
From this energy bound first observe that
\[
\int_\O \chi_k\, dx \leq \frac{M\e_k}{\kappa} \to 0 \quad \text{as }k \to \infty
\]
which shows that $\chi=0$ a.e.\ in $\O$.

\medskip

Since $\mathbf{A}_s\xi:\xi\geq  c_s|\xi|^2$ and $\mathbf{A}_w\xi:\xi\geq  c_w |\xi|^2$, Young's inequality yields
$$W_\e(\xi) \geq \min\left\{\frac{ c_s}{2}|\xi|^2,\sqrt{\frac{2\eta_\e\kappa  c_w}{\e}}|\xi| \right\}.$$
If $\eta_\e/\e \to \alpha \in (0,+\infty]$, then we can find a constant $c>0$, only depending on $ c_w$, $ c_s$, $\kappa$, and $\alpha$, such that
\begin{equation}\label{eq:boundbelow}
W_\e(\xi) \geq c |\xi|-\frac{1}{c} \quad \text{ for all }\xi \in \Ms.
\end{equation}
As a consequence, we have
$$c \int_\O |e(u_k)|\, dx -\frac{\LL^n(\O)}{c}\leq \int_\O W_{\e_k}(e(u_k))\, dx \leq E_{\e_k}(u_k,\chi_k) \leq M.$$
This implies that the sequence $(u_k)_{k \in \N}$ is bounded in $BD(\O)$, and thus $u_k \wto u$ weakly* in $BD(\O)$ with $u \in BD(\O)$.
\end{proof}

\section{The Hencky regime}\label{Hencky}

In this section, we consider the case $\alpha \in (0,\infty)$. Our main result reads as follows.

\begin{thm}\label{thm:alpha=1}
Let $\O \subset \R^n$ ($n=2$ or $n=3$) be a bounded open set with Lipschitz boundary. Assume that $\mathbf{A}_w$ and $\mathbf{A}_s$ are isotropic tensors, i.e., for all $\xi \in \Ms$,  
\begin{align*}
\mathbf{A}_w \xi & = \lambda_w (\tr\xi)\, \Id+ 2 \mu_w \xi,\\
\mathbf{A}_s \xi & = \lambda_s (\tr\xi)\, \Id+ 2 \mu_s \xi,
\end{align*}
where $\lambda_i>0$ and $\mu_i>0$ are the Lam\'e coefficients. If
\[
  \alpha := \lim_{\e \to 0}\frac{\eta_\e}{\e} \in (0,\infty),
\]
then the functionals $\E_\e$ $\Gamma$-converge as $\e \to 0$ with respect to the strong $L^1(\O;\R^n) \times L^1(\O)$-topology to the functional $\E_0:L^1(\O;\R^n) \times L^1(\O) \to [0,+\infty]$ defined by
\[
\E_0(u,\chi):=
\begin{cases}
\ds   \int_\O \overline W(e(u))\, dx + \int_\O \sqrt{2\alpha\kappa h\left( \frac{dE^s u}{d|E^s u|} \right)}\, d|E^s u|& \text{ if } \chi=0\text{ a.e.\ and } u \in BD(\O),\\
+\infty & \text{ otherwise.}
\end{cases}
\]
Here, the limit integrand is given by the infimal convolution
\[
\overline W:=  f \, \Box \, \sqrt{2\alpha\kappa h},
\]
where, if $\xi_1 \leq \cdots \leq \xi_n$ denote the ordered eigenvalues of $\xi \in \Ms$,
\begin{equation}\label{eq:fh}
f(\xi) := \frac12 \mathbf{A}_s\xi:\xi, \quad h(\xi) :=\mu_w \left(\sum_{i=1}^n|\xi_i| \right)^2+(\lambda_w+\mu_w) \left(\sum_{i=1}^n \xi_i \right)^2.
\end{equation}
\end{thm}

\begin{rem} \label{rem:DPR}
According to~\cite[Theorem 1.7]{DPR}, if $u \in BD(\O)$, then for $|E^s u|$-a.e.\ $x \in \O$, there exist $a(x), b(x) \in \R^n \setminus \{0\}$ such that 
$$\frac{dE^s u}{d|E^s u|}(x)=a(x) \odot b(x).$$
Therefore, also using Proposition~\ref{prop:W_alpha} below, the $\Gamma$-limit $\E_0(u,\chi)$ for $\chi=0$ a.e.\ and $u \in BD(\O)$ can alternatively be expressed as 
\[
\E_0(u,\chi) = \int_\O \overline W(e(u))\, dx + \sqrt{2\alpha \kappa} \int_\O \sqrt{\mathbf{A}_w \frac{dE^s u}{d|E^s u|}:\frac{dE^s u}{d|E^s u|} }\, d|E^s u|.
\]
\end{rem}

\subsection{Explanatory examples}

Before addressing the proof of Theorem~\ref{thm:alpha=1}, let us explain the appearance of the term $\sqrt{2 \alpha\kappa h}$ in $\overline W$, in the simplified case where $\Omega=Q=(0,1)^2$ is a cube in $\R^2$, $\eta_\e=\e$, and $u$ is an affine function.

\medskip

\noindent {\it Case 1:} Let $u(x)=\xi x$, where $\xi \in \mathbb M^{2 \times 2}_{\rm sym}$ is a diagonal matrix whose eigenvalues $\xi_1$ and $\xi_2$ satisfy $\xi_1\xi_2>0$. We consider integers $N_\e \in \N$  such that $N_\e \to +\infty$ as $\e \to 0$, and we subdivide the interval $(0,1)$ into $N_\e + 1$ sub-intervals of length $1/(N_\e+1)$. For each $i = 0, 1, \ldots, N_\e+1$, we define $s_i^\e:=i/(N_\e+1)$. For $j=1,2$ we choose
$$\delta_\e^j:=\frac{|\xi_j| \sqrt{\mathbf{A}_w (e_j \otimes e_j):(e_j \otimes e_j)}}{2\sqrt{2\kappa}} \cdot \frac{\e}{N_\e+1}$$
and set
$$u^{j}_\e(s):=
\begin{cases}
\xi_j s_{i-1}^\e + \xi_j \frac{s_i^\e- s_{i-1}^\e}{2\delta^j_\e}(s-s_i^\e+\delta^j_\e) & \text{ if }
\begin{cases}
s_i^\e-\delta_\e^j \leq s \leq s_i^\e+\delta^j_\e,\\
1\leq i \leq N_\e,
\end{cases}
\\
\xi_j s_i^\e & \text{ if }
\begin{cases}
s_{i}^\e+\delta_\e^j <s<s_{i+1}^\e-\delta^j_\e,\\
0\leq i \leq N_\e,
\end{cases}
\end{cases}
$$ 
and $u^j_\e$ is extended as a constant up to the boundary of $[0,1]$. We also introduce the sets
$$\Delta^{j}_\e:=\bigcup_{i=1}^{N_\e} (s_i^\e-\delta^j_\e,s_i^\e+\delta^j_\e), \quad D_\e^{j}:=\{x \in Q : \; x_j \in \Delta_\e^{j}\},$$
satisfying $\LL^1(\Delta^{j}_\e)=2N_\e\delta^j_\e \to 0$ as $\e \to 0$. Finally, we define the displacement and the damaged set by
$$u_\e(x):=(u_\e^{1}(x_1),u_\e^{2}(x_2)) \quad \text{ for all } x \in Q  \quad \text{ and }\quad D_\e:=D^{1}_\e \cup D^{1}_\e.$$
Note that $u_\e \to  u$ in $L^2(Q;\R^2)$ and $\LL^2(D_\e)\to 0$. We also observe that
\[
  e(u_\e)(x) = \sum_{j=1}^2 (u_\e^{j})'(x_j) e_j \otimes e_j  \quad\text{for a.e.\ } x \in D_\e;
\]
in particular, $e(u_\e)(x) = 0$ for a.e.\ $x \in Q \setminus D_\e$. Therefore, 
\begin{align*}
&\int_{\{x_1\in \Delta^{1}_\e,\ x_2\notin \Delta^{2}_\e\}}\Big(\frac\e2 \mathbf{A}_w e(u_\e):e(u_\e)+\frac{\kappa}{\e}\Big)\, dx\\
&\qquad =\int_{\{x_1\in \Delta^{1}_\e,\ x_2\notin \Delta^{2}_\e\}}\Big(\frac\e2 \mathbf{A}_w (u_\e^{1})'(x_1) e_1 \otimes e_1: (u_\e^{1})'(x_1) e_1 \otimes e_1+\frac{\kappa}{\e}\Big)\, dx\\
&\qquad =\int_{\{x_1\in \Delta^{1}_\e,\ x_2\notin \Delta^{2}_\e\}}\Big(\frac\e2\Big(\frac{\xi_1}{2\delta^1_\e (N_\e+1)}\Big)^2 \mathbf{A}_w (e_1 \otimes e_1): (e_1 \otimes e_1)+\frac{\kappa}{\e}\Big)\, dx\\
&\qquad \leq 2\delta^1_\e N_\e \Big(\frac\e2\Big(\frac{\xi_1}{2\delta^1_\e (N_\e+1)}\Big)^2 \mathbf{A}_w (e_1 \otimes e_1):( e_1 \otimes e_1)+\frac{\kappa}{\e}\Big)\\
&\qquad =\frac{\xi_1^2\e N_\e}{4\delta^1_\e (N_\e+1)^2} \mathbf{A}_w (e_1 \otimes e_1):( e_1 \otimes e_1)+
\frac{2\kappa \delta^1_\e N_\e}{\e}\\
&\qquad \leq |\xi_1| \sqrt{2\kappa\mathbf{A}_w (e_1 \otimes e_1):( e_1 \otimes e_1)}.
\end{align*}
A similar computation can be performed to show that
$$\int_{\{x_2\in \Delta^{2}_\e,\ x_1\notin \Delta^{1}_\e\}}\Big(\frac\e2 \mathbf{A}_w e(u_\e):e(u_\e)+\frac{\kappa}{\e}\Big)\, dx \leq |\xi_2| \sqrt{2\kappa\mathbf{A}_w (e_2 \otimes e_2): (e_2 \otimes e_2)}.$$
Finally, we have that
$$\int_{\{x_1\in \Delta^{1}_\e,\ x_2\in \Delta^{2}_\e\}}\Big(\frac\e2 \mathbf{A}_w e(u_\e):e(u_\e)+\frac{\kappa}{\e}\Big)\, dx \to 0.$$
We conclude that
$$\limsup_{\e \to 0} E_\e(u_\e,\chi_{D_\e})\leq \sum_{j=1}^2|\xi_j| \sqrt{2\kappa\mathbf{A}_w (e_j \otimes e_j):( e_j \otimes e_j)}=\sqrt{2\kappa(\lambda_w+2\mu_w)}\sum_{j=1}^2|\xi_j|=\sqrt{2\kappa h(\xi)}$$
since the eigenvalues have the same sign.

\medskip

\noindent {\it Case 2:} Let now $\xi \in \mathbb M^{2 \times 2}_{\rm sym}$ be a diagonal matrix and assume that its eigenvalues satisfy $\xi_1 \xi_2 \leq 0$. Then, according to Lemma~\ref{lem:aodotb}, we have $\xi=a \odot b$ for some $a,b \in \R^2$. Let us consider the linear function
\[
u(x)= a (x \cdot b) \quad \text{ for all }x \in Q
\]
and notice that $e(u)=\xi$. Using the same notation as before, but setting this time
\[
\delta_\e:=\frac{\sqrt{h(a \odot b)}}{2\sqrt{2\kappa}} \cdot \frac{\e}{N_\e+1},
\]
we define
\[
w_\e(s):=
\begin{cases}
s_{i-1}^\e + \frac{s_i^\e- s_{i-1}^\e}{2\delta_\e}(s-s_i^\e+\delta_\e) & \text{ if }
\begin{cases}
s_i^\e-\delta_\e <s<s_i^\e+\delta_\e,\\
1\leq i\leq N_\e,
\end{cases}\\
s_i^\e & \text{ if }
\begin{cases}
s_{i}^\e+\delta_\e<s<s_{i+1}^\e-\delta_\e,\\
0\leq i \leq N_\e,
\end{cases}
\end{cases}
\]
and $w_\e$ is extended as a constant up to the boundary of $[0,1]$.
The displacement is now given by
\[
u_\e(x) := a w_\e(x \cdot b),
\]
while the damaged set is defined by
\[
D_\e:=\{x \in Q :\;  x \cdot a \in \Delta_\e\},  \qquad\text{where}\qquad
\Delta_\e:=\bigcup_{i=1}^{N_\e} (s_i^\e-\delta_\e,s_i^\e+\delta_\e).
\]
Again we have $u_\e \to u$ in $L^2(Q;\R^2)$ and $\LL^2(D_\e) \to 0$. 
Observe that $e(u_\e)(x)= (a \odot b) w_\e'(x \cdot b)$ for a.e.\ $x \in Q$ and so $e(u_\e)(x)=0$ for a.e.\ $x \in Q \setminus D_\e$. Then, from Proposition~\ref{prop:W_alpha} below we have $\mathbf{A}_w (a \odot b) : (a \odot b) = h(a \odot b)$, and so
\begin{align*}
E_\e(u_\e,\chi_{D_\e})
&=\int_{D_\e}\Big(\frac\e2 \mathbf{A}_w [w_\e'(x \cdot b) a \odot b] : [w_\e'(x \cdot b) a \odot b] +\frac{\kappa}{\e}\Big) \, dx\\
&=\int_{D_\e}\Big(\frac\e2\Big(\frac{1}{2\delta_\e (N_\e+1)}\Big)^2 h(a \odot b) + \frac{\kappa}{\e}\Big) \,dx\\
&\leq 2\delta_\e N_\e \Big(\frac\e2\Big(\frac{1}{2\delta_\e (N_\e+1)}\Big)^2 h(a \odot b) +\frac{\kappa}{\e}\Big)+o(1)\\
&=\frac{\e N_\e}{4\delta_\e (N_\e+1)^2} h(a \odot b) +
\frac{2\kappa \delta_\e N_\e}{\e}+o(1)\\
&=\sqrt{2\kappa h(a \odot b)}+o(1).
\end{align*}

In both cases, these explicit constructions show that $\sqrt{2\kappa h(\xi)}$ is an upper bound for the $\Gamma$-limit in the concentrating zone, at least when $u$ is affine and $e(u)$ is diagonal. This suggests that $\sqrt{2\kappa h(\xi)}$ will describe the (linear) slope at infinity of the effective energy density.

\subsection{Pointwise limit of relaxed energy densities}

We next investigate the pointwise properties of the functions $W^\e$. Let us denote by $SQW_\e$ the symmetric quasiconvex envelope of $W_\e$ given by
$$SQW_\e(\xi):=\inf_{\varphi \in \C^\infty_c((0,1)^n;\R^n)} \int_{(0,1)^n} W_\e(\xi+e(\varphi))\, dx,\quad \xi\in \Ms.$$
From~\cite[Proposition 5.2]{AL}, we know that it can be expressed as 
$$SQW_\e(\xi)=\min_{0 \leq \theta \leq 1} F_\e(\theta,\xi),$$
where 
\begin{align*}
F_\e(\theta,\xi)&:= \frac{\eta_\e}{2}\mathbf{A}_w\xi:\xi + \frac{\kappa\theta}{\e}+ (1-\theta) \max_{\tau \in \Ms} \left\{\tau:\xi -\frac12 (\mathbf{A}_s-\eta_\e\mathbf{A}_w)^{-1}\tau:\tau -\frac{\theta}{2\eta_\e}G(\tau) \right\}\\
&\phantom{:}= \frac{\eta_\e}{2}\mathbf{A}_w\xi:\xi + \frac{\kappa\theta^2}{\e}+ (1-\theta) \max_{\tau \in \Ms} \left\{\tau:\xi -\frac12 (\mathbf{A}_s-\eta_\e\mathbf{A}_w)^{-1}\tau:\tau +\frac{\theta}{2\eta_\e} \left(\frac{2\kappa\eta_\e}{\e}-G(\tau)\right) \right\}
\end{align*}
and, if $\tau_1 \leq \cdots \leq \tau_n$ are the ordered eigenvalues of $\tau$,
\begin{equation}\label{d:G}G(\tau):=
\begin{cases}
\frac{\tau_1^2}{\lambda_w+2\mu_w}& \text{ if } \frac{\lambda_w+2\mu_w}{2(\lambda_w+\mu_w)}(\tau_1+\tau_n)<\tau_1,\\
\frac{(\tau_1-\tau_n)^2}{4\mu_w}+\frac{(\tau_1+\tau_n)^2}{4(\lambda_w+\mu_w)} & \text{ if }  \tau_1 \leq \frac{\lambda_w+2\mu_w}{2(\lambda_w+\mu_w)}(\tau_1+\tau_n)\leq \tau_n,\\
\frac{\tau_n^2}{\lambda_w+2\mu_w}& \text{ if } \tau_n < \frac{\lambda_w+2\mu_w}{2(\lambda_w+\mu_w)}(\tau_1+\tau_n).
\end{cases}
\end{equation}
As it is remarked in~\cite{AL} (below Proposition 5.2 in \textit{loc.~cit.}), the maximization above is over a strictly concave function, so a maximizer indeed exists.

In the following result we identify the poinwise limit $\overline W$ of $SQW_\e$, which turns out to be a density typically encountered in plasticity theory, i.e.\ a quadratic function close to the origin and with linear growth at infinity.

\begin{prop}\label{prop:Walpha}
Setting $K:=\bigl\{\tau \in \Ms: \; G(\tau) \leq 2\alpha\kappa\bigr\}$, we have 
$$SQW_\e \to \overline W:=(f^*+I_{K})^*$$
pointwise on $\Ms$. 
\end{prop}

\begin{proof}
Fix $\xi \in \Ms$. Let us first prove that, as $\e \to 0$, $(F_\e(\cdot,\xi))_{\e>0}$ $\Gamma$-converges in $[0,1]$ to the function $F_0(\cdot,\xi)$ defined by $F_0(\theta,\xi):=\overline W(\xi)$ if $\theta=0$ and $F_0(\theta,\xi):=+\infty$ if $\theta\neq 0$.

\medskip

\noindent {\it Lower bound:} Let $(\theta_\e)_{\e>0}$ be a sequence in $[0,1]$. If $\liminf_\e F_\e(\theta_\e,\xi)=+\infty$, there is nothing to prove. Without loss of generality, we can therefore assume that $\liminf_\e F_\e(\theta_\e,\xi)<+\infty$. Moreover, up to a subsequence, we can also suppose that the previous lower limit is actually a limit, and that $\theta_\e\to \theta \in [0,1]$. Since $F_\e(\theta_\e,\xi) \geq \frac{\kappa\theta_\e}{\e}$ (choose $\tau = 0$), we deduce that $\theta=0$. We next estimate $F_\e$ from below as follows: for all $\tau \in \Ms$,
$$F_\e(\theta_\e,\xi)
\geq (1-\theta_\e) \left\{ \tau:\xi -\frac12 (\mathbf{A}_s-\eta_\e\mathbf{A}_w)^{-1}\tau:\tau +\frac{\theta_\e}{2\eta_\e} \left(\frac{2\kappa\eta_\e}{\e}-G(\tau)\right) \right\}.$$
Let $\tau \in K$, i.e.\ $G(\tau) \leq 2\alpha \kappa$. For every $\e$, we define $\tau_\e:=\sqrt{\frac{\eta_\e}{\alpha \e}}\tau$, for which ($G$ being $2$-homogeneous) $G(\tau_\e) \leq 2\kappa\eta_\e/\e$ and $\tau_\e \to \tau$.  Specifying the previous inequality to $\tau_\e$, we get that
$$F_\e(\theta_\e,\xi) \geq (1-\theta_\e) \left\{ \tau_\e:\xi -\frac12 (\mathbf{A}_s-\eta_\e\mathbf{A}_w)^{-1}\tau_\e:\tau_\e\right\}.$$
Passing to the limit as $\e \to 0$, and using that $\tau$ is arbitrary in $K$, we deduce that
$$\liminf_{\e \to 0}F_\e(\theta_\e,\xi) \geq \sup_{\tau \in K} \left\{\tau:\xi -\frac12 \mathbf{A}_s^{-1}\tau:\tau\right\}=(f^*+I_{K})^*(\xi)=\overline W(\xi).$$

\medskip 

\noindent {\it Upper bound:} If $\theta \neq 0$, the proof is immediate. We can thus assume without loss of generality that $\theta=0$. Let $\lambda\geq 0$ and set $\theta_\e:=\lambda\eta_\e \to 0$. Then, since $(\mathbf{A}_s-\eta_\e\mathbf{A}_w)^{-1} \geq \mathbf{A}_s^{-1}$ as quadratic forms,
\begin{align*}
F_\e(\theta_\e,\xi) &\leq \frac{\eta_\e}{2}\mathbf{A}_w\xi:\xi + \frac{\kappa\lambda^2\eta_\e^2}{\e}+\sup_{\tau \in \Ms} \left\{\tau:\xi - \frac12 \mathbf{A}_s^{-1}\tau:\tau+ \frac{\lambda}{2} \left(\frac{2\kappa\eta_\e}{\e} -G(\tau)\right)  \right\}\\
 &= \frac{\eta_\e}{2}\mathbf{A}_w\xi:\xi + \frac{\kappa\lambda^2\eta_\e^2}{\e}+\sup_{\tau \in \Ms} \left\{\tau:\xi - \frac12 \mathbf{A}_s^{-1}\tau:\tau+ \frac{\lambda}{2} \left(2\kappa\alpha -G(\tau)\right)  \right\}+\lambda\kappa\Big(\frac{\eta_\e}{\e}-\alpha\Big).
\end{align*}
Passing to the limit as $\e \to 0$ and then taking the infimum with respect to $\lambda\geq 0$, we get
$$\limsup_{\e \to 0} F_\e(\theta_\e,\xi) \leq \inf_{\lambda\geq 0}\sup_{\tau \in \Ms} \left\{\tau:\xi - \frac12 \mathbf{A}_s^{-1}\tau:\tau+ \frac{\lambda}{2} \big(2\alpha\kappa -G(\tau)\big)  \right\}.$$
{According to standard results on inequality-constrained optimization problems (see, e.g.,~\cite[Chapter VI, Proposition 2.3]{ET}), we have (note that the function inside the curly braces is concave in $\tau$ and affine in $\lambda$)
\begin{align*}
&\inf_{\lambda\geq 0}\sup_{\tau \in \Ms} \left\{\tau:\xi- \frac12 \mathbf{A}_s^{-1}\tau:\tau- \frac{\lambda}{2} \big(G(\tau)-2\alpha\kappa\big)  \right\}\\
 &\qquad =  \sup_{\tau \in \Ms} \inf_{\lambda\geq 0} \left\{\tau:\xi- \frac12 \mathbf{A}_s^{-1}\tau:\tau- \frac{\lambda}{2} \big(G(\tau)-2\alpha\kappa\big)  \right\} \\
 &\qquad = \sup_{\tau \in K} \left\{\tau:\xi-\frac12 \mathbf{A}_s^{-1}\tau:\tau\right\},
\end{align*}
 from which we deduce that
$$\limsup_{\e \to 0} F_\e(\theta_\e,\xi) \leq  \sup_{\tau \in K} \left\{\tau:\xi-\frac12 \mathbf{A}_s^{-1}\tau:\tau\right\}=(f^*+I_{K})^*(\xi)=\overline W(\xi).$$
}
\medskip

\noindent {\it Convergence of minimizers.} According to the fundamental theorem of $\Gamma$-convergence, we deduce that
$$SQW_\e(\xi)=\min_{0 \leq \theta \leq 1}F_\e(\theta,\xi) \to \min_{0 \leq \theta \leq 1}F_0(\theta,\xi)=\overline W(\xi),$$
which completes the proof of the proposition.
\end{proof}

The following result relates the function $h$ to the convex conjugate of the indicator function of the closed convex set $K$.

\begin{lem}\label{l:h}
For all $\xi \in \Ms$,
$$I_{K}^*(\xi)=\sqrt{2\alpha\kappa h(\xi)},$$
where $h$ is defined in~\eqref{eq:fh}. In particular, $\overline W=f \, \Box \, \sqrt{2\alpha\kappa h}$.
\end{lem}

\begin{proof}
For all $\xi \in \Ms$, we have 
$$I_{K}^*(\xi)=\sup_{\tau \in K}\tau:\xi = \sup_{k(\tau) \leq 1} \tau:\xi=k^\circ(\xi),$$
where $k(\tau):=\sqrt{G(\tau)/2\alpha\kappa}$ and $k^\circ$ is the polar function of $k$.  The function  $k$ is a nonnegative, real-valued, lower semicontinuous, and positively $1$-homogeneous function such that $k(0)=0$. According to the terminology of~\cite[Section 15]{R} $k$ is a closed gauge, and thanks to~\cite[Corollary 15.3.1]{R}, we get that
$$\frac12 (I_{K}^*)^2=\frac12 (k^\circ)^2=\left(\frac12 k^2\right)^*=\left(\frac{1}{4\alpha\kappa} G\right)^*.$$
From~\cite[Proof of Theorem 5.3]{AL} we have that
$$h(\xi)=\sup_{\tau \in \Ms} \bigl\{ 2\tau:\xi - G(\tau) \bigr\}=G^*(2\xi),$$
and since $h$ is $2$-homogeneous, $G^*=\frac14 h$. We thus infer that
$$\left(\frac{1}{4\alpha\kappa} G\right)^*(\xi)=\sup_{\tau \in \Ms}\left( \tau:\xi - \frac{1}{4\alpha\kappa}G(\tau)\right)=\frac{1}{4\alpha\kappa}G^*(4\alpha\kappa \xi)=\alpha\kappa h(\xi),$$
where we used again the fact that $h$ is $2$-homogeneous. We thus deduce that $I_{K}^*=\sqrt{2\alpha\kappa h}$.
\end{proof}

\begin{rem}
We observe that the function $\sqrt{2\alpha\kappa h}$ can also be considered as the pointwise limit of the symmetric quasiconvex envelope of the generalized Kohn--Strang functional (see~\cite{KS}), defined by
$$\bar g_\e (\xi):=
\begin{cases}
\frac{\eta_\e}{2}\mathbf{A}_w\xi:\xi + \frac{\kappa}{\e} & \text{ if }\xi \neq 0,\\
0 &  \text{ if }\xi = 0.
\end{cases}$$
Indeed, according to~\cite[Theorem 5.3]{AL}, the symmetric quasiconvex envelope of $\bar g_\e$ can be explicitely computed, namely
$$SQ\bar g_\e(\xi)=
\begin{cases}
\frac{\eta_\e}{2}\mathbf{A}_w\xi:\xi + \frac{\kappa}{\e}& \text{ if } h(\xi) \geq \frac{2\kappa}{\eta_\e\e},\\
\sqrt{\frac{2\eta_\e\kappa h(\xi)}{\e}} + \frac{\eta_\e}{2}(\mathbf{A}_w\xi:\xi - h(\xi)) & \text{ if } h(\xi) < \frac{2\kappa}{\eta_\e\e},
\end{cases}
$$
and so we observe that $SQ\bar g_\e \to \sqrt{2\alpha\kappa h}$ pointwise on $\Ms$.
\end{rem}

We are now in the position to prove several properties of the energy density $\overline W$.

\begin{prop}\label{prop:W_alpha}
The function $\overline W$ is convex,
\begin{equation}\label{eq:growthW_alpha}
c|\xi| - \frac{1}{c}\leq \overline W(\xi) \leq C|\xi| \quad \text{ for all }\xi \in \Ms,
\end{equation}
for some $c,C>0$, and 
\begin{equation}\label{eq:LipschitzW_alpha}
|\overline W(\xi_1) - \overline W(\xi_2)|\leq L|\xi_1-\xi_2| \quad \text{  for all }\xi_1, \; \xi_2 \in \Ms,
\end{equation}
for some $L>0$.
In addition, its recession function, defined for all $\xi \in \Ms$ by
$$\overline W^\infty(\xi):=\lim_{t \to +\infty}\frac{\overline W(t\xi)}{t},$$
exists and is given by
$$\overline W^\infty(\xi)=\sqrt{2\alpha\kappa h(\xi)}.$$
Finally, for all $a$, $b \in \R^n$,
$$\overline W^\infty(a \odot b)=\sqrt{2\alpha\kappa \mathbf{A}_w (a \odot b):(a \odot b)}.$$
\end{prop}

\begin{proof}
The function $\overline W=(f^*+I_{K})^*$ is convex and lower semicontinuous as the supremum of affine functions. Moreover, since $f^*+I_{K} \geq I_{K}$, we get that $\overline W \leq I_{K}^*=\sqrt{2\alpha\kappa h}$. Hence, for all $\xi \in \Ms$,
$$\overline W(\xi) \leq C |\xi|$$
for some $C>0$. Concerning the bound from below, according to~\eqref{eq:boundbelow} we have
$$\overline W(\xi)=\lim_{\e \to 0}SQW_\e(\xi) \geq \limsup_{\e \to 0} W^{**}_\e(\xi) \geq c|\xi| - \frac{1}{c},$$
which shows the validity of the growth and coercivity conditions~\eqref{eq:growthW_alpha}. Then, as $\overline W$ is a convex function with linear growth, it is in particular globally Lipschitz (see, e.g.,~\cite[Lemma 5.6]{Ri}) which shows the validity of~\eqref{eq:LipschitzW_alpha}.

Note that the convexity of $\overline W$ together with $\overline W(0)=0$ implies that, for all $\xi \in \Ms$, 
$$t \mapsto \frac{\overline W(t\xi)}{t}$$
is increasing, and thus that the limit as $t \to +\infty$ exists. The recession function is thus well defined on $\Ms$. In particular, since $\overline W \leq \sqrt{2 \alpha \kappa h}$ and since the latter function is positively $1$-homogeneous, we infer that  $\overline W^\infty \leq \sqrt{2 \alpha \kappa h}$. To prove the converse inequality, we use that $\overline W = f \, \Box \, I_{K}^*= f \, \Box \, \sqrt{2\alpha\kappa h}$. Then, by definition of inf-convolution, for all $t >0$, there exists some $\xi'_t \in \Ms$ such that
$$ \frac{\overline W(t\xi)}{t}= \frac{f(t\xi - t\xi'_t)}{t}+\frac{\sqrt{2\alpha\kappa h(t\xi'_t)}}{t}.$$
Since $f$ and $h$ are $2$-homogeneous, we get that
$$ \frac{\overline W(t\xi)}{t}= t f(\xi - \xi'_t)+\sqrt{2\alpha\kappa h(\xi'_t)}.$$
Using the growth condition~\eqref{eq:growthW_alpha} and the coercivity of the tensor $\mathbf{A}_s$, we have
$$\frac{ c_s}{2} t|\xi-\xi'_t|^2 \leq t f(\xi - \xi'_t) \leq t f(\xi - \xi'_t)+\sqrt{2\alpha\kappa h(\xi'_t)}=\frac{\overline W(t\xi)}{t} \leq C |\xi|,$$
proving that $\xi'_t \to \xi$ as $t \to +\infty$. Therefore, by continuity of $h$,
$$\overline W^\infty(\xi)=\lim_{t \to +\infty} \frac{\overline W(t\xi)}{t} \geq \limsup_{t \to +\infty}\sqrt{2\alpha\kappa h(\xi'_t)}=\sqrt{2\alpha\kappa h(\xi)},$$
which shows that $\overline W^\infty=\sqrt{2\alpha\kappa h}$.

Finally, if $\xi = a\odot b$, let us denote by $\xi_1, \ldots , \xi_n$ its eigenvalues. If $\xi$ has only one nonzero eigenvalue (say $\xi_1$), then
$$\sum_{i=1}^n|\xi_i|=|\xi_1| = \left|\sum_{i=1}^n\xi_i\right|,$$
which implies in view of~\eqref{eq:fh} that $h(\xi)=\mathbf{A}_w\xi:\xi$. If $\xi$ has two nonzero eigenvalues (say $\xi_1$ an $\xi_2$, we know from Lemma~\ref{lem:aodotb} that they must have opposite signs, hence (also using that $|\xi|^2 = \xi : \xi = \xi_1^2 + \xi_2^2$)
$$h(\xi)-\mathbf{A}_w\xi:\xi=2\mu_w (\xi_1\xi_2 + |\xi_1| |\xi_2|)=0,$$
which completes the proof of the proposition.
\end{proof}

\subsection{Proof of Theorem~\ref{thm:alpha=1}}

\begin{proof}
\noindent {\it Step 1: The upper bound.} We first assume that $u \in W^{1,\infty}(\O;\R^n)$. According to the dominated convergence theorem, we infer that
$$\int_\O \overline W(e(u))\, dx = \lim_{\e \to 0}\int_\O SQW_\e(e(u))\, dx.$$
For every $\e>0$,
$$v \in W^{1,1}(\O;\R^n) \mapsto \int_\O SQW_\e(e(v))\, dx$$
 is the $ L^1(\O;\R^n)$-lower semicontinuous envelope, restricted to $W^{1,1}(\O,\R^n)$, of
 $$v \in W^{1,1}(\O;\R^n) \mapsto \int_\O W_\e(e(v))\, dx,$$
see~\cite{BFT,ARDPR}. It is thus possible to find a recovery sequence $(u^\e_k)_{k \in \N} \subset W^{1,1}(\O;\R^n)$ such that
	$$\int_\O SQW_\e(e(u))\, dx=\lim_{k \to \infty} \int_\O W_\e(e(u^\e_k))\, dx.$$
Using a diagonalization argument, we extract a subsequence $k(\e) \to \infty$ as $\e \to 0$ such that $v_\e:=u_{k(\e)}^\e \to u$ in $L^1(\O;\R^n)$ and
$$\int_\O \overline W(e(u))\, dx=\lim_{\e \to 0} \int_\O W_\e(e(v_\e))\, dx.$$
Then, defining the damaged sets as
$$D_\e:=\left\{x \in \O : \; (\mathbf{A}_s-\eta_\e\mathbf{A}_w)e(v_\e)(x):e(v_\e)(x) \geq \frac{2\kappa}{\e}\right\},$$
we obtain by construction that
$$\E_0''(u,0) \leq \limsup_{\e \to 0} E_\e(v_\e,\chi_{D_\e})= \lim_{\e \to 0}\int_\O W_\e(e(v_\e))\, dx= \int_\O \overline W(e(u))\, dx.$$

Since $\Omega$ has a Lipschitz boundary, according to the density result~\cite[Proposition I.1.3]{Temam}, the previous inequality can be extended to any $u \in LD(\O)$. Indeed, let $(u_k)_{k \in \N}$ be a sequence in $W^{1,\infty}(\O;\R^n)$ such that $u_k \to u$ in $LD(\O)$. By lower semicontinuity of $\E''_0(\cdot,0)$ with respect to the $L^1(\O;\R^n)$ topology, and by continuity of  
$$LD(\O) \ni v \mapsto \int_\O \overline W(e(v))\, dx,$$
we deduce that
$$\E''_0(u,0)\leq \liminf_{k \to +\infty} \E''_0(u_k,0) \leq \int_\O \overline W(e(u_k))\, dx=\int_\O \overline W(e(u))\, dx.$$

Finally, if $u \in BD(\O)$, according to the relaxation result proved in~\cite[Corollary 1.10]{ARDPR}, we can find a sequence $(v_k)_{k \in \N}$ in $LD(\O)$ such that $v_k \to u$ in $L^1(\O;\R^n)$ and
$$ \int_\O \overline W( e(v_k))\, dx \to  \int_\O \overline W(e(u))\, dx+\int_\O \overline W^\infty\left(\frac{dE^s u}{d|E^s u|} \right) d|E^s u|.$$
Using again the lower semicontinuity of $\E''_0(\cdot,0)$ with respect to the $L^1(\O;\R^n)$ topology, we infer that
$$\E''_0(u,0)\leq \int_\O \overline W(e(u))\, dx + \int_\O \overline W^\infty\left(\frac{dE^s u}{d|E^s u|} \right) d|E^s u|=\E_0(u,0),$$
which completes the proof of the upper bound.

\medskip

\noindent {\it Step 2: The lower bound.} Let $(u_\e,\chi_\e)_{\e>0}$ be a sequence in $L^1(\O;\R^n) \times L^1(\O)$ such that $u_\e \to u \in BD(\O)$ in $L^1(\O;\R^n)$ and $\chi_\e \to 0$ in $L^1(\O)$. According to (the proof of) Lemma~\ref{lem:comp1} and the fact that $\eta_\e/\e \to \alpha \in (0,+\infty)$, we infer that 
\begin{equation}\label{eq:bnd}
\sup_{\e>0} \, \left\{ \|e(u_\e)\|_{L^1(\O)}+ \eta_\e \|e(u_\e)\|^2_{L^2(\O)}\right\} <+\infty.
\end{equation}
Let $v_\e:=\sqrt{\eta_\e} u_\e$. By the energy estimates~\eqref{eq:bnd} and Korn's and Poincar\'e's inequalities, this sequence is bounded
 in $H^1(\O;\R^n)$, hence $v_\e \wto 0$ weakly in $H^1(\O;\R^n)$. 
 
For every open set $\omega \subset \O$, let us define the set function
 $$\mu(\omega):=\liminf_{\e \to 0}\left\{ \frac12 \int_\omega[\eta_\e \chi_\e \mathbf{A}_w + (1-\chi_\e)\mathbf{A}_s]e(u_\e):e(u_\e)\, dx + \frac{\kappa}{\e} \int_\omega \chi_\e\, dx\right\} ,$$
 which is clearly a super-additive set function on disjoint open sets, i.e.\ $\mu(\omega_1 \cup \omega_2) \geq \mu(\omega_1) + \mu(\omega_2)$ for all open sets $\omega_1,\, \omega_2 \subset \O$, with $\overline \omega_1 \cap \overline \omega_2 =\emptyset$ and $\overline \omega_1 \cup \overline \omega_2 \subset \O$.
 
\medskip

\noindent {\it Step 2a: The two-dimensional case.} For all $r \in [0,1]$, we have by Young's inequality (see also~\eqref{d:fg}) for all $\xi \in \mathbb M^{2 \times 2}_{\rm sym}$,
\begin{align}\label{eq:0809}
g_\e(\xi) &= \frac{\eta_\e}{2} \bigl(\mathbf{A}_w \xi:\xi +4\mu_w r \det(\xi)\bigr)+\frac{\kappa}{\e}-2\mu_w{\eta_\e} r \det(\xi)\nonumber\\
&\geq \sqrt{2\kappa{\frac{\eta_\e}{\e}}\bigl(\mathbf{A}_w \xi:\xi+4\mu_wr \det(\xi)\bigr)}-2\mu_w{\eta_\e} r \det(\xi)\nonumber\\
&\geq \sqrt{2\alpha\kappa h_r(\xi)} -2\mu_w{\eta_\e} r \det(\xi)+ o(1)|\xi|,
\end{align}
where $o(1) \to 0$ as $\e \to 0$ and
$$h_r(\xi):=\mathbf{A}_w \xi:\xi+4\mu_wr \det(\xi) \quad \text{ for all }\xi \in \mathbb M^{2 \times 2}_{\rm sym}.$$
Note that since $0 \leq r \leq 1$ and $2 |\det(\xi)| \leq |\xi|^2$, we deduce that $h_r$ is a nonnegative quadratic form, and thus the function $\sqrt{2\alpha\kappa h_r}$ is convex.

We next claim that for every $\gamma > 0$ there exists $\e_0>0$ (depending on $\gamma$) such that for all $r \in [0,1]$ and all $\e\leq \e_0$,
$$-2\mu_w{\eta_\e} r \det(\xi) \leq \frac{\gamma}{2} \mathbf{A}_s \xi:\xi = \gamma f(\xi)\quad \text{ for all }\xi \in \mathbb M^{2 \times 2}_{\rm sym}.$$
Indeed, if $\det(\xi) \geq 0$ the result is obvious, while if $\det(\xi)<0$, then
using that $-2\det(\xi) \leq |\xi|^2$, we have
$$-2\mu_w{\eta_\e r} \det(\xi) \leq \mu_w {\eta_\e} |\xi|^2 \leq \frac{\gamma c_s}{2} |\xi|^2 \leq  \frac{\gamma}{2} \mathbf{A}_s \xi:\xi,
$$
provided we choose $\e_0 > 0$ such that for $\e\leq \e_0$ we have $\eta_\e\leq \gamma c_s/(2\mu_w)$. Thus, for $\e \leq \e_0$, we have
\begin{equation}\label{eq:gamma}
f(\xi) \geq (1-\gamma)f(\xi)-2\mu_w{\eta_\e} r \det(\xi) \quad \text{ for all }\xi \in \mathbb M^{2 \times 2}_{\rm sym}
\end{equation}
and gathering \eqref{eq:0809} together with \eqref{eq:gamma}, yields
\begin{equation}\label{eq:wedge}
(f\wedge g_\e)(\xi) \geq \left(\big(1-\gamma)f\big) \wedge \sqrt{2\alpha\kappa h_r}\right)(\xi)-2\mu_w{\eta_\e} r \det(\xi)+o(1)|\xi|\quad \text{ for all }\xi \in \mathbb M^{2 \times 2}_{\rm sym}.
\end{equation}

Let $\omega \subset \O$ be an open set. Then, for all $\varphi \in \mathcal C_c(\omega)$ with $0 \leq \varphi \leq 1$ and all $\e \leq \e_0$, we obtain using \eqref{eq:wedge} that
\begin{align*}
 &\frac12 \int_\omega \Bigl[\eta_\e \chi_\e \mathbf{A}_w + (1-\chi_\e)\mathbf{A}_s \Bigr]e(u_\e):e(u_\e)\, dx + \frac{\kappa}{\e} \int_\omega \chi_\e\, dx\\
  &\qquad \geq \int_\omega \varphi \, \Bigl[(1-\chi_\e) f(e(u_\e)) + \chi_\e g_\e (e(u_\e)) \Bigr]\, dx\\
  &\qquad \geq \int_\omega \varphi \, (f \wedge g_\e)(e(u_\e)) \, dx \\
   &\qquad \geq \int_\omega \varphi\left(\big(1-\gamma)f\big) \wedge \sqrt{2\alpha\kappa h_r}\right)(e(u_\e)) \, dx -2\mu_w\eta_\e r \int_\omega \varphi\, \det(e(u_\e)) \, dx + o(1) \int_\omega |e(u_\e)|\, dx.
\end{align*}
Since $v_\e \wto 0$ weakly in $H^1(\O;\R^2)$, then $\det\nabla v_\e \wto 0$ weakly* in $\M(\O)$, see~\cite[Theorem~8.20]{D}. On the other hand, since $\eta_\e\det(e(u_\e))=\det(e(v_\e)) \leq \det(\nabla v_\e)$ by Young's inequality, we infer that
 $$\limsup_{\e \to 0} \eta_\e\int_\omega \varphi \, \det(e(u_\e))\, dx \leq \lim_{\e \to 0} \int_\omega \varphi\, \det\nabla v_\e\, dx = 0.$$
Therefore, using that $o(1) \to 0$ and that $(e(u_\e))_{\e>0}$ is bounded in $L^1(\O;\mathbb M^{2 \times 2}_{\rm sym})$,
\begin{align*}
\mu(\omega) &\geq (1-\gamma) \liminf_{\e \to 0} \int_\omega \varphi \, (f \wedge \sqrt{2\alpha\kappa h_r})(e(u_\e))\, dx \\
&\geq (1-\gamma) \liminf_{\e \to 0}\int_\omega \varphi \, (f\,\Box\, \sqrt{2\alpha\kappa h_r})(e(u_\e))\, dx.
\end{align*}
Since $f\,\Box\, \sqrt{2\alpha\kappa h_r}$ is convex, $(x,\xi) \mapsto \varphi(x) (f\,\Box\, \sqrt{2\alpha\kappa h_r})(\xi)$ is continuous, and
$$0 \leq \varphi(x) (f\,\Box\, \sqrt{2\alpha\kappa h_r})(\xi) \leq C(1+|\xi|) \quad \text{ for all }(x,\xi) \in \omega \times \mathbb M^{2 \times 2}_{\rm sym},$$
for some constant $C>0$, a standard weak* lower semicontinuity result for convex functionals of measures shows that
\begin{align*}
&\liminf_{\e \to 0} \int_\omega \varphi \, (f \,\Box\, \sqrt{2\alpha\kappa h_r})(e(u_\e))\, dx\\
&\qquad\geq  \int_\omega \varphi \,   (f\,\Box\, \sqrt{2\alpha\kappa h_r})(e(u))\, dx +\int_\omega \varphi (f\,\Box\, \sqrt{2\alpha\kappa h_r})^\infty \left(\frac{dE^s u}{d|E^su|}\right)d|E^su|.
\end{align*}
Also letting $\gamma \to 0$, we thus infer that
$$\mu(\omega) \geq\int_\omega \varphi \, (f\,\Box\, \sqrt{2\alpha\kappa h_r})(e(u))\, dx +\int_\omega \varphi (f\,\Box\, \sqrt{2\alpha\kappa h_r})^\infty \left(\frac{dE^s u}{d|E^su|}\right)d|E^su|,$$
and passing to the supremum with respect to all $\varphi \in \mathcal C_c(\omega)$ with $0 \leq \varphi \leq 1$, yields
\begin{equation} \label{eq:mu_omega_2D}
\mu(\omega) \geq\int_\omega  (f\,\Box\, \sqrt{2\alpha\kappa h_r})(e(u))\, dx +\int_\omega (f\,\Box\, \sqrt{2\alpha\kappa h_r})^\infty \left(\frac{dE^s u}{d|E^su|}\right)d|E^su|.
\end{equation}

In order to pass to the supremum with respect to $r \in [0,1]$, let us observe that for all $\xi \in \mathbb M^{2 \times 2}_{\rm sym}$,
$$\max_{r \in [0,1]} h_r(\xi)=\max_{r \in \{0,1\}} h_r(\xi)=\mathbf{A}_w \xi:\xi+4\mu_w (\det(\xi))^+=h(\xi).$$
For fixed $\xi \in \mathbb M^{2 \times 2}_{\rm sym}$, we have that $\xi' \in \mathbb M^{2 \times 2}_{\rm sym} \mapsto f(\xi-\xi') + \sqrt{2\alpha\kappa h_r(\xi')}$ is convex, continuous and coercive, while $r \in [0,1] \mapsto  f(\xi-\xi') + \sqrt{2\alpha\kappa h_r(\xi')}$ is concave and continuous. According to~\cite[Chapter VI, Proposition 2.3]{ET}), we get that
\begin{align*}
\sup_{r \in [0,1]} (f\,\Box\, \sqrt{2\alpha\kappa h_r})(\xi)& = \sup_{r \in [0,1]} \inf_{\xi' \in \mathbb M^{2 \times 2}_{\rm sym}} \left\{ f(\xi-\xi') + \sqrt{2\alpha\kappa h_r(\xi')}\right\}\\ 
& =  \inf_{\xi' \in \mathbb M^{2 \times 2}_{\rm sym}} \sup_{r \in [0,1]} \left\{ f(\xi-\xi') + \sqrt{2\alpha\kappa h_r(\xi')}\right\}\\ 
& =  \inf_{\xi' \in \mathbb M^{2 \times 2}_{\rm sym}} \left\{ f(\xi-\xi') + \sqrt{2\alpha\kappa h(\xi')}\right\}\\ 
& = (f\,\Box\, \sqrt{2\alpha\kappa h})(\xi).
\end{align*}
In addition, since, for $r \in [0,1]$, the functions $f\,\Box\, \sqrt{2\alpha\kappa h_r}$ and $f\,\Box\, \sqrt{2\alpha\kappa h}$ are convex, and $(f\,\Box\, \sqrt{2\alpha\kappa h_r})(0)=(f\,\Box\, \sqrt{2\alpha\kappa h})(0)=0$, we get that
\begin{align*}
\sup_{r \in [0,1]} (f\,\Box\, \sqrt{2\alpha\kappa h_r})^\infty(\xi) &= \sup_{r \in [0,1]} \sup_{t>0} \frac{(f\,\Box\, \sqrt{2\alpha\kappa h_r})(t\xi)}{t}\\
& = \sup_{t>0} \sup_{r \in [0,1]} \frac{(f\,\Box\, \sqrt{2\alpha\kappa h_r})(t\xi)}{t}\\
& = \sup_{t>0}  \frac{(f\,\Box\, \sqrt{2\alpha\kappa h})(t\xi)}{t}\\
& =  (f\,\Box\, \sqrt{2\alpha\kappa h})^\infty(\xi).
\end{align*}
Thus, applying~\cite[Proposition 1.16]{Braides} to~\eqref{eq:mu_omega_2D}, we obtain
$$\mu(\omega) \geq\int_\omega  (f\,\Box\, \sqrt{2\alpha\kappa h})(e(u))\, dx+\int_\omega (f\,\Box\, \sqrt{2\alpha\kappa h})^\infty \left(\frac{dE^s u}{d|E^su|}\right)d|E^su|.$$
Hence, also using Lemma~\ref{l:h},
$$\liminf_{\e \to 0} E_\e(u_\e,\chi_\e)=\mu(\O)\geq  \int_\O \overline W(e(u))\, dx+\int_\O \overline W^\infty \left(\frac{dE^s u}{d|E^su|}\right)d|E^su|,$$
whereby $\E_0'(u,0) \geq \E_0(u,0)$.

\medskip

\noindent {\it Step 2b: The three-dimensional case.} By direct computation we obtain, for all $\xi \in  \mathbb M^{3 \times 3}_{\rm sym}$, 
\begin{equation} \label{eq:h3D}
  h(\xi)-\mathbf{A}_w \xi:\xi=4\mu_w \bigl( (\xi_1\xi_2)^+ + (\xi_1\xi_3)^+ + (\xi_2 \xi_3)^+ \bigr),
\end{equation}
where $\xi_1$, $\xi_2$, and $\xi_3$ are the eigenvalues of $\xi \in  \mathbb M^{3 \times 3}_{\rm sym}$. According to Lemma~\ref{lem:polyconvex}, $\xi_1\xi_2$, $\xi_1\xi_3$ and $\xi_2\xi_3$ are the eigenvalues of $\cof(\xi)$, and we observe that at least one of them is nonnegative. The highest eigenvalue of $\cof(\xi)$ can be computed as the maximum of the Rayleigh quotient
$$\lambda_{\max}(\cof(\xi)):=\max_{|y|=1} \cof(\xi)y \cdot y \geq 0.$$
The other two eigenvalues of $\cof(\xi)$ have the same sign. We can thus write that
\begin{align*}
(\xi_1\xi_2)^+ + (\xi_1\xi_3)^+ + (\xi_2 \xi_3)^+ & = \max \, \bigl\{ \lambda_{\max}(\cof(\xi)),\tr(\cof(\xi)) \bigr\}\\
& = \max_{|y|=1} \max \, \bigl\{\cof(\xi)y\cdot y,\tr(\cof(\xi)) \bigr\}.
\end{align*}
Let us define the following set of matrices:
$$M:= \bigl\{A \in \mathbb M^{3 \times 3}_{\rm sym} : \; A=\Id \text{ or } A=y \otimes y \text{ for }y \in \R^3 \text{ with }|y|=1 \bigr\}.$$ 
Since $\cof(\xi)y \cdot y=\cof(\xi):(y \otimes y)$ and $\tr(\cof(\xi))=\cof(\xi):\Id$, the previous argument shows that for all $\xi \in  \mathbb M^{3 \times 3}_{\rm sym} $,
\begin{equation}\label{eq:lambdamax}
(\xi_1\xi_2)^+ + (\xi_1\xi_3)^+ + (\xi_2 \xi_3)^+=\max_{A \in M} \bigl\{A : \cof(\xi)\bigr\} = \max_{A \in {\rm conv}(M)} \bigl\{A :\cof(\xi)\bigr\},
\end{equation}
where in the last equality we denote by ${\rm conv}(M)$ the convex hull of $M$, which is a closed set. This last equality then follows since the mapping $A \mapsto A:\cof(\xi)$ is linear.

For all $A \in {\rm conv}(M)$, we define the quadratic form
$$h_A(\xi):=\mathbf{A}_w \xi:\xi +4\mu_w A:\cof(\xi),  \qquad \xi \in  \mathbb M^{3 \times 3}_{\rm sym}.$$
We claim that for all $A \in {\rm conv}(M)$, the quadratic form $h_A$ is convex. Indeed, on the one hand, if $A=\Id$, the function $h_{\Id}:\xi \mapsto \mathbf{A}_w\xi:\xi +4\mu_w\tr(\cof(\xi))=(\lambda_w+2\mu_w)(\tr(\xi))^2$ is clearly a convex quadratic form. On the other hand, let us consider a matrix $A=y \otimes y$ for some $y \in \R^3$ with $|y|=1$. Let us write $\xi=P D P^T$ where $P \in SO(3)$ and $D={\rm diag}(\xi_1,\xi_2,\xi_3)$, so that, according to Lemma~\ref{lem:polyconvex}, we have $\cof(\xi)=P \cof(D)P^T$, where $\cof(D)={\rm diag}(\xi_2\xi_3,\xi_1\xi_3,\xi_1\xi_2)$. We have that the quadratic form $h_{y \otimes y}:\xi \mapsto \mathbf{A}_w\xi:\xi +4\mu_w\cof(\xi)y\cdot y$ can be written in the basis of the eigenvectors of $\xi$ as
\begin{align*}
h_{y\otimes y}(\xi) &= \lambda_w(\xi_1+\xi_2+\xi_3)^2 +2\mu_w (\xi_1^2+\xi_2^2+\xi_3^2)\\
&\qquad + 4\mu_w(P^Ty)^2_1\xi_2\xi_3+4\mu_w(P^Ty)^2_2\xi_1\xi_3+4\mu_w(P^Ty)^2_3\xi_1\xi_2.
\end{align*}
If $\xi_1\xi_2\geq 0$, $\xi_2\xi_3\geq 0$, and $\xi_1 \xi_3\geq 0$, then the previous expression is clearly nonnegative. Otherwise, there exists exactly one nonnegative eigenvalues of $\cof(D)$ and both the other eigenvalues are nonpositive. Up to a permutation of indices, there is no loss of generality in assuming that $\xi_1\xi_2 \geq 0$, $\xi_2\xi_3 \leq 0$, and $\xi_1 \xi_3 \leq 0$. 
For simplicity, we define $z:=P^Ty$.  Using Young's inequality and that $|z|=1$, we get that
\begin{align*}
h_{y\otimes y}(\xi) & = \lambda_w(\xi_1+\xi_2+\xi_3)^2 +2\mu_w (\xi_1^2+\xi_2^2+\xi_3^2)- 4\mu_w z^2_1 |\xi_2\xi_3|-4\mu_wz^2_2|\xi_1\xi_3|+4\mu_wz^2_3|\xi_1\xi_2|\\
&\geq  \lambda_w(\xi_1+\xi_2+\xi_3)^2 +2\mu_w (\xi_1^2+\xi_2^2+\xi_3^2)\\
&\qquad - 2\mu_w z^2_1 (\xi_2^2+\xi^2_3)-2\mu_wz^2_2(\xi_1^2+\xi_3^2)+4\mu_wz^2_3|\xi_1\xi_2|\\
& = \lambda_w(\xi_1+\xi_2+\xi_3)^2 +4\mu_wz^2_3|\xi_1\xi_2| +2\mu_w (1-z_2^2)\xi_1^2+2\mu_w(1-z_1^2)\xi_2^2+2\mu_w z_3^2\xi_3^2 \\
&\geq  0.
\end{align*}
Since the mapping $A \mapsto h_A(\xi)$ is linear, we deduce that also if $A \in {\rm conv}(M)$, then the quadratic forms $h_{A}$ are nonnegative. Thus, the functions $\sqrt{2\alpha\kappa h_A}$ are convex for all $A \in {\rm conv}(M)$.

We can then proceed in a similar fashion to the two-dimensional case. Note that for all $\gamma > 0$ there exists $\e_0>0$ such that, for all $A \in {\rm conv}(M)$ and all $\e\leq \e_0$, we have
$$-2\mu_w\eta_\e A:\cof(\xi)\leq \frac{\gamma}{2} \mathbf{A}_s \xi:\xi = \gamma f(\xi) \quad \text{ for all }\xi \in \mathbb M^{3 \times 3}_{\rm sym}.$$
As a consequence, for all open sets $\omega \subset \O$, all $\varphi \in \mathcal C_c(\omega)$ with $0 \leq \varphi \leq 1$, and all $A \in {\rm conv}(M)$, we get (via Young's inequality)
\begin{align*}
g_\e(\xi) &= \frac{\eta_\e}{2} \bigl(\mathbf{A}_w \xi:\xi +4\mu_w A:\cof(\xi)\bigr)+\frac{\kappa}{\e}-2\eta_\e\mu_w A:\cof(\xi)\\
&\geq \sqrt{2\alpha\kappa h_A(\xi)} -2\eta_\e\mu_w A:\cof(\xi)+ o(1) |\xi|.
\end{align*}
where $o(1) \to 0$ as $\e \to 0$. Thus,
\begin{align*}
 &\frac12 \int_\omega \bigl[ \eta_\e \chi_\e \mathbf{A}_w + (1-\chi_\e)\mathbf{A}_s \bigr]e(u_\e):e(u_\e)\, dx + \frac{\kappa}{\e} \int_\omega \chi_\e\, dx\\
  &\qquad\geq  \int_\omega \varphi \, \bigl[(1-\chi_\e) f(e(u_\e)) + \chi_\e g_\e (e(u_\e))\bigr]\, dx\\
&\qquad\geq (1-\gamma) \int_\omega \varphi \, (f \wedge \sqrt{2\alpha\kappa h_A})(e(u_\e)) \, dx -2\mu_w\eta_\e  \int_\omega \varphi\, A:\cof(e(u_\e))\, dx +o(1) \int_\omega|e(u_\e)|\, dx.
\end{align*}
Let $F \in\mathbb M^{3 \times 3}$. According to linear algebra manipulations (see, e.g.,~\cite[Eq.~(3.2)]{BKS}), we have 
$$\cof(F^{\rm sym})=(\cof(F))^{\rm sym} - \cof(F^{\rm skew}),$$
where $\cof(F^{\rm skew})$ is a nonnegative matrix (see, e.g.,~\cite[Eq.~(3.4)]{BKS}). Thus, for all $y\in \R^3$, we get
\begin{align*}
\cof(F^{\rm sym}) y\cdot y &\leq (\cof(F))^{\rm sym}y \cdot y=\cof(F) y\cdot y, \\
\tr(\cof(F^{\rm sym})) &\leq  \tr((\cof(F))^{\rm sym})=\tr(\cof(F)),
\end{align*}
which implies that
\begin{equation}\label{cof1}
A:\cof(F^{\rm sym}) \leq A:\cof(F) \quad \text{ for all } A \in {\rm conv}(M).
\end{equation}
Since $v_\e := \sqrt{\eta_\e} u_\e \wto 0$ weakly in $H^1(\O;\R^3)$, then $\cof(\nabla v_\e) \wto 0$ weakly* in $\M(\O;\mathbb M^{3 \times 3}_{\rm sym})$, see~\cite[Theorem~8.20]{D}. Therefore,~\eqref{cof1} implies that 
\begin{align*}
\limsup_{\e \to 0}\eta_\e \int_{\omega}\varphi\,A: \cof(e(u_{\e}))\, dx
&=\limsup_{\e \to 0}\int_{\omega}\varphi\, A: \cof(e(v_{\e}))\, dx  \notag\\
&\leq \lim_{\e \to 0}\int_{\omega}\varphi\, A: \cof(\nabla v_{\e})\, dx\\
&= 0.  \label{cof}
\end{align*}
Hence,
\begin{align*}
\mu(\omega) &\geq (1-\gamma)  \liminf_{\e \to 0} \int_\omega \varphi \, (f \wedge \sqrt{2\alpha\kappa h_A})(e(u_\e))\, dx \\
&\geq (1-\gamma)  \liminf_{\e \to 0}\int_\omega \varphi \, (f\,\Box\, \sqrt{2\alpha\kappa h_A})(e(u_\e))\, dx.
\end{align*}
Since $f\,\Box\, \sqrt{2\alpha\kappa h_A}$ is convex, $(x,\xi) \mapsto \varphi(x) (f\,\Box\, \sqrt{2\alpha\kappa h_A})(\xi)$ is continuous, and
$$0 \leq \varphi(x) (f\,\Box\, \sqrt{2\alpha\kappa h_A})(\xi) \leq C(1+|\xi|) \quad \text{ for all }(x,\xi) \in \omega \times \mathbb M^{3 \times 3}_{\rm sym},$$
for some constant $C>0$, a standard weak* lower semicontinuity result for convex functionals of measures shows that
\begin{align*}
 &\liminf_{\e \to 0} \int_\omega \varphi \, (f \,\Box\, \sqrt{2\alpha\kappa h_A})(e(u_\e))\, dx\\
 &\qquad \geq  \int_\omega \varphi \, (f\,\Box\, \sqrt{2\alpha\kappa h_A})(e(u))\, dx +\int_\omega \varphi (f\,\Box\, \sqrt{2\alpha\kappa h_A})^\infty \left(\frac{dE^s u}{d|E^su|}\right)d|E^su|.
 \end{align*}
Also letting $\gamma \to 0$, we thus infer that
$$\mu(\omega) \geq\int_\omega \varphi \, (f\,\Box\, \sqrt{2\alpha\kappa h_A})(e(u))\, dx +\int_\omega \varphi (f\,\Box\, \sqrt{2\alpha\kappa h_A})^\infty \left(\frac{dE^s u}{d|E^su|}\right)d|E^su|,$$
and passing to the supremum with respect to all $\varphi \in \mathcal C_c(\omega)$ with $0 \leq \varphi \leq 1$, yields
$$\mu(\omega) \geq\int_\omega  (f\,\Box\, \sqrt{2\alpha\kappa h_A})(e(u))\, dx +\int_\omega (f\,\Box\, \sqrt{2\alpha\kappa h_A})^\infty \left(\frac{dE^s u}{d|E^su|}\right)d|E^su|.$$
It thus remains to pass to the supremum with respect to $A \in {\rm conv}(M)$. Let us observe that, according to{~\eqref{eq:h3D}}~\eqref{eq:lambdamax}, for all $\xi \in \mathbb M^{3 \times 3}_{\rm sym}$,
\begin{equation}\label{eq:h_A}
\max_{A \in {\rm conv}(M)} h_A(\xi)=\mathbf{A}_w \xi:\xi+4\mu_w \max\{\lambda_{\max}(\cof(\xi)),\tr(\cof(\xi))\}=h(\xi).
\end{equation}
We claim that
$$(f\,\Box\, \sqrt{2\alpha\kappa h})(\xi)=\max_{A \in {\rm conv}(M)}  (f\,\Box\, \sqrt{2\alpha\kappa h_A})(\xi).$$
Indeed, the set ${\rm conv}(M)$ is compact and convex, and, for fixed $\xi \in \mathbb M^{3 \times 3}_{\rm sym}$, we have that $\xi' \in \mathbb M^{3 \times 3}_{\rm sym} \mapsto f(\xi-\xi') + \sqrt{2\alpha\kappa h_A(\xi')}$ is convex, continuous and coercive, while $A \in {\rm conv}(M) \mapsto  f(\xi-\xi') + \sqrt{2\alpha\kappa h_A(\xi')}$ is concave and continuous. Then,~\cite[Chapter VI, Proposition 2.3]{ET}) ensures that
\begin{align*}
\sup_{A \in {\rm conv}(M)}  (f\,\Box\, \sqrt{2\alpha\kappa h_A})(\xi) & = \sup_{A \in {\rm conv}(M)} \inf_{\xi' \in \mathbb M^{3 \times 3}_{\rm sym}} \left\{ f(\xi-\xi') + \sqrt{2\alpha\kappa h_A(\xi')}\right\}\\
& = \inf_{\xi' \in \mathbb M^{3 \times 3}_{\rm sym}} \sup_{A \in {\rm conv}(M)}\left\{ f(\xi-\xi') + \sqrt{2\alpha\kappa h_A(\xi')}\right\}\\
& =  \inf_{\xi' \in \mathbb M^{3 \times 3}_{\rm sym}} \left\{f(\xi-\xi') +\sqrt{2\alpha\kappa h(\xi')} \right\}\\ 
& = (f\,\Box\, \sqrt{2\alpha\kappa h})(\xi).
\end{align*}
where we used~\eqref{eq:h_A} in the second-to-last equality. In addition, since, for $A \in {\rm conv}(M)$, the functions $f\,\Box\, \sqrt{2\alpha\kappa h_A}$ and $f\,\Box\, \sqrt{2\alpha\kappa h}$ are convex, and $(f\,\Box\, \sqrt{2\alpha\kappa h_A})(0)=(f\,\Box\, \sqrt{2\alpha\kappa h})(0)=0$, we get that 
\begin{align*}
\sup_{A \in {\rm conv}(M)} (f\,\Box\, \sqrt{2\alpha\kappa h_A})^\infty(\xi) &=  \sup_{A \in {\rm conv}(M)} \sup_{t>0} \frac{(f\,\Box\, \sqrt{2\alpha\kappa h_A})(t\xi)}{t}\\
& = \sup_{t>0} \sup_{A \in {\rm conv}(M)} \frac{(f\,\Box\, \sqrt{2\alpha\kappa h_A})(t\xi)}{t}\\
& = \sup_{t>0}  \frac{(f\,\Box\, \sqrt{2\alpha\kappa h})(t\xi)}{t}\\
& =  (f\,\Box\, \sqrt{2\alpha\kappa h})^\infty(\xi).
\end{align*}
Finally, using~\cite[Proposition 1.16]{Braides} as before and also invoking Lemma~\ref{l:h}, we get that
$$\liminf_{\e \to 0} E_\e(u_\e,\chi_\e) =\mu(\O) \geq \int_\O \overline W(e(u))\, dx+\int_\O \overline W^\infty\left(\frac{dE^s u}{d|E^su|} \right)d|E^su|,$$
and so $\E_0'(u,0) \geq \E_0(u,0)$.
\end{proof}

The next result (which is not used anywhere else) establishes a relaxation-type formula for the effective energy density $\overline W$ in the spirit of~\cite{BDV1,BDV}.
{\begin{prop} \label{prop:Walpha_char}
For all $\xi \in \Ms$, we have
\begin{align*}
\overline W=\sup\Big\{\varphi : \Ms \to \R \text{ convex, }&\varphi(\xi) \leq f(\xi) \text{ for all }\xi \in \Ms,\\
&\varphi(a \odot b) \leq \sqrt{2\alpha\kappa \mathbf{A}_w(a\odot b):(a \odot b)} \text{ for all }a,b \in \R^n\Big\}.
\end{align*}
\end{prop}

\begin{proof}
According to Proposition~\ref{prop:Walpha} and Lemma~\ref{l:h}, we can write 
$$\overline W=(f^*+I_{K})^*=f \, \Box\, \sqrt{2\alpha\kappa h}=(f \wedge \sqrt{2\alpha\kappa h})^{**}.$$ 
Therefore, if we prove that the convex envelope of the function $H:\Ms \to [0,+\infty]$ defined by
$$H(\xi):=
\begin{cases}
\sqrt{2\alpha\kappa \mathbf{A}_w \xi:\xi} & \text{ if }\xi=a\odot b \text{ for some }a,b \in \R^n,\\
+\infty & \text{ otherwise,}
\end{cases}$$
is given by $\sqrt{2\alpha\kappa h}$, we then may conclude $\overline W=(f\wedge (H^{**}))^{**}=(f\wedge H)^{**}$, that is, the conclusion of the proposition. First of all, since by Proposition~\ref{prop:W_alpha} we have $H(a\odot b)=\sqrt{2\alpha\kappa h(a \odot b)}$ for all $a$, $b \in \R^n$, we get that $\sqrt{2\alpha\kappa h}\leq H$, and since $\sqrt{2\alpha\kappa h}$ is convex, we get $\sqrt{2\alpha\kappa h}\leq H^{**}$.

We now establish the reverse inequality $\sqrt{2\alpha\kappa h}\geq H^{**}$, which is equivalent to $I_K \leq H^*$, i.e., $H^*(\tau)=+\infty$ for all $\tau \notin K$. So, let us fix $\tau \notin K$, i.e.\ $G(\tau)>2\alpha\kappa$ where $G$ is given by~\eqref{d:G}. Since all expressions of matrices only depend on the eigenvalues, it is not restrictive to assume that $\tau$ is diagonal with ordered eigenvalues $\tau_1\leq\dots\leq\tau_n$.

We distinguish three cases.

\medskip	
\noindent {\it Case I:} If
$$\frac{\lambda_w+2\mu_w}{2(\lambda_w+\mu_w)}(\tau_1+\tau_n) < \tau_1,$$ then according to~\eqref{d:G}, we have that $2\alpha\kappa<G(\tau)=\frac{\tau_1^2}{\lambda_w+2\mu_w}$. 

The computation of the convex conjugate of $H$ gives
\begin{align*}
H^*(\tau) &=\sup_{t>0} \sup_{|a|=|b|=1} t \left\{ \tau:( a \odot b) - \sqrt{2\alpha \kappa \mathbf{A}_w (a\odot b):(a\odot b)}\right\}\\
&= \sup_{t>0} \sup_{|a|=|b|=1} t \left\{(\tau a) \cdot b - \sqrt{2\alpha\kappa \big((\lambda_w+\mu_w)(a\cdot b)^2 +\mu_w \big)}\right\}.
\end{align*}
In order to show that $H^*(\tau)=+\infty$, it is enough to prove that 
$$M:=\max_{|a|=|b|=1} \left\{(\tau a) \cdot b - \sqrt{2\alpha\kappa \big((\lambda_w+\mu_w)(a\cdot b)^2 +\mu_w \big)}\right\}>0.$$

Taking $a=e_1$ and $b=\pm e_1$, we deduce that
$$M \geq |\tau_1| - \sqrt{2\alpha \kappa(\lambda_w+2\mu_w)}>0.$$

\medskip
	
\noindent {\it Case II:} If 
\begin{equation}\label{thirdregime}
\tau_1 \leq \frac{\lambda_w+2\mu_w}{2(\lambda_w+\mu_w)}(\tau_1+\tau_n) \leq \tau_n,	
\end{equation} 
then according to~\eqref{d:G}, we have that 
\[ \label{Gthirdregime}
2\alpha\kappa<G(\tau)=\frac{(\tau_1-\tau_n)^2}{4\mu_w}+\frac{(\tau_1+\tau_n)^2}{4(\lambda_w+\mu_w)}.	
\]
We will rewrite $H^*(\tau)$ in a more convenient form. Denoting by $R$ the set of the diagonal $n \times n$ matrices of the form $\xi = a \odot b$ ($a,b \in \R^n$) with ordered eigenvalues $\xi_1\leq0=\xi_2=\dots=\xi_{n-1}=0 \leq\xi_n$ (see Lemma~\ref{lem:aodotb}), we have
\begin{equation}\label{R1}
H^*(\tau) \geq\sup_{\xi\in R} \left\{ \tau:\xi - \sqrt{2\alpha \kappa \mathbf{A}_w \xi:\xi}\right\}.
\end{equation} 
Let us set
$$\tau_s:=\tau_n+\tau_1,\qquad \tau_d:=\tau_n-\tau_1,$$
so that $\tau_1=(\tau_s-\tau_d)/2$, $\tau_n=(\tau_s+\tau_d)/2$, and~\eqref{thirdregime},\eqref{R1} become
\begin{equation}\label{G2}
2\alpha\kappa<G(\tau)=\frac{\tau_d^2}{4\mu_w}+\frac{\tau_s^2}{4(\lambda_w+\mu_w)},\qquad \frac{\mu_w}{\lambda_w+\mu_w}|\tau_s| \leq \tau_d,
\end{equation} 
\begin{equation}\label{R2}
H^*(\tau) \geq\sup_{|\xi_s|\leq\xi_d} \left\{ \frac{\tau_s \xi_s}{2}+\frac{\tau_d\xi_d}{2} - \sqrt{2\alpha \kappa} \Big((\lambda_w+\mu_w)\xi_s^2+\mu_w\xi_d^2\Big)^{1/2}\right\}.
\end{equation} 
Changing the variables to
$$\tilde \xi_s:=\sqrt{\lambda_w+\mu_w}\xi_s,\quad \tilde{\xi}_d:=\sqrt{\mu_w}\xi_d,\quad \tilde{\tau}_s:=\frac{\tau_s}{2\sqrt{\lambda_w+\mu_w}},\quad \tilde{\tau}_d:=\frac{\tau_d}{2\sqrt{\mu_w}},$$
equations~\eqref{G2} and~\eqref{R2} become
\begin{equation}\label{G3}
2\alpha\kappa<{\tilde\tau_d^2}+{\tilde\tau_s^2},\qquad \sqrt{\frac{\mu_w}{\lambda_w+\mu_w}}|\tilde\tau_s| \leq \tilde\tau_d,
\end{equation} 
\begin{equation}\label{R3}
H^*(\tau) \geq\sup_{\sqrt{\frac{\mu_w}{\lambda_w+\mu_w}}|\tilde\xi_s|\leq\tilde\xi_d} \left\{{\tilde\tau_s \tilde\xi_s}+{\tilde\tau_d\tilde\xi_d} - \sqrt{2\alpha \kappa} \Big(\tilde\xi_s^2+\tilde\xi_d^2\Big)^{1/2}\right\}.
\end{equation} 
Finally, introducing the vectors $x,y\in\R^2$ given as
$$x:=(\tilde \tau_s,\tilde\tau_d),\qquad y:=(\tilde\xi_s,\tilde{\xi_d}),$$
equations~\eqref{G3},~\eqref{R3} reduce to
$$2\alpha\kappa<|x|^2,\qquad \sqrt{\frac{\mu_w}{\lambda_w+\mu_w}}|x_1| \leq x_2,	$$
$$H^*(\tau) \geq\sup_{\sqrt{\frac{\mu_w}{\lambda_w+\mu_w}}|y_1|\leq y_2} \left\{x\cdot y - \sqrt{2\alpha \kappa} |y|\right\}=+\infty,$$
choosing $y=tx$, $t>0$.

\medskip

\noindent {\it Case III:} If
$$\tau_n < \frac{\lambda_w+2\mu_w}{2(\lambda_w+\mu_w)}(\tau_1+\tau_n),$$
then according to~\eqref{d:G}, we have that $2\alpha\kappa<G(\tau)=\frac{\tau_n^2}{\lambda_w+2\mu_w}$. Repeating the computations of Case I and taking $a=e_n$ and $b=\pm e_n$, we deduce that
$$M \geq |\tau_n| - \sqrt{2\alpha \kappa(\lambda_w+2\mu_w)}>0.$$
This concludes the proof.
\end{proof}

\section{The trivial regime}

We now treat the first of the endpoint cases.

\begin{thm}\label{thm:alpha=0}
Let $\O \subset \R^n$ be a bounded open set and let $\mathbf{A}_w$, $\mathbf{A}_s$ be fourth-order symmetric elasticity tensors satisfying~\eqref{eq:growth}. If $\alpha=0$ in~\eqref{eq:alpha}, then the functionals $\E_\e$ $\Gamma$-converge as $\e \to 0$ with respect to the strong $L^1(\O;\R^n) \times L^1(\O)$-topology to the functional
$\Phi_0:L^1(\O;\R^n) \times L^1(\O) \to [0,+\infty]$ defined by
$$\Phi_0(u,\chi)=
\begin{cases}
0 & \text{ if $\chi=0$ a.e.\ in $\O$,}\\
+\infty & \text{ otherwise.}
\end{cases}$$
\end{thm}

\begin{proof}
Clearly, the lower bound $\E_0'(u,\chi) \geq \Phi_0(u,\chi)$ holds for all $(u,\chi) \in L^1(\O;\R^n) \times L^1(\O)$. On the other hand, it is enough to prove the upper bound  $\E''_0(u,\chi)=0$ whenever $\chi=0$ a.e.\ in $\O$, since $\Phi_0$ is infinite otherwise. We assume for simplicity by translating and rescaling that $\O \subset Q:=(0,1)^n$. We extend $u$ by zero in $Q \setminus \O$ so that the extension (still denoted by $u$) belongs to $L^1(Q;\R^n)$. 

\medskip

\noindent {\it Step 1.} We first assume that  $u$ is of the form
\begin{equation}\label{eq:pwcst}
u=\sum_{i\in \{0,\ldots,N-1\}^n}u_i \chi_{Q_i},
\end{equation}
where $u_i \in \R^n$ for all $i \in \{0,\ldots,N-1\}^n$ and $\{Q_i\}_{i \in \{0,\ldots,N-1\}^n}$ is a subdivision of $Q$ (up to an $\LL^n$-negligible set) into $N^n$ open cubes 
$$Q_i:=\frac1N(i+Q)$$
of side length $1/N$ with $N \in \N$, and $i \in \{0,\ldots,N-1\}^n$. Therefore, up to a set of zero Lebesgue measure, we have
$$Q=\bigcup_{i\in \{0,\ldots,N-1\}^n} Q_i.$$

Since $\eta_\e \ll \e$, one can find a sequence $(\delta_\e)_{\e>0}$ such that $\eta_\e \ll \delta_\e \ll \e$ (meaning $\lim_{\e\to 0} \eta_\e/\delta_\e = \lim_{\e\to 0} \delta_\e/\e = 0$). 
We denote by $Q^{1-\delta_\e}=(1-\delta_\e)Q$ the cube concentric with $Q$, having side length $1-\delta_\e$.
Let $\varphi_\e \in \mathcal C_c^\infty(\R^n;[0,1])$ be a cut-off function such that $\varphi_\e \equiv 1$ on $Q^{1-\delta_\e}$, $\varphi_\e \equiv 0$ on $\R^n \setminus Q$, $0 < \varphi_\e < 1$ on $Q \setminus \overline{Q^{1-\delta_\e}}$, and $|\nabla \varphi_\e|\leq C/\delta_\e$. We then define the displacement $u_\e \in L^1(Q;\R^n)$ by
$$u_\e(x):=\sum_{i \in \{0,\ldots,N-1\}^n} u_i \varphi_\e\left(Nx-i\right)\quad \text{ for all }x \in Q,$$
and the damaged set by
$$ D_\e:=\bigcup_{i \in \{0,\ldots,N-1\}^n} \bigl\{x \in Q_i:\; 0<\varphi_\e(Nx-i)<1\bigr\}.$$

Note that  $u_\e \in H^1(Q;\R^n)$, and since $\varphi_\e \to \chi_{Q}$ in $L^1(\R^n)$ we have $u_\e \to u$ in $L^1(Q;\R^n)$. In addition,
$$e(u_\e)(x)=N \sum_{i \in \{0,\ldots,N-1\}^n} u_i \odot \nabla \varphi_\e\left(Nx-i\right)\quad \text{ for all }x \in Q,$$
and since $u_\e$ is constant in each connected component of $Q\setminus D_\e$, we infer that
$$e(u_\e)(x)=0 \quad \text{ for all }x \in Q \setminus D_\e.$$
We also remark that
$$\LL^n(D_\e)=\sum_{i \in \{0,\ldots,N-1\}^n} \left(\frac{1}{N}\right)^n\Big(1-(1-\delta_\e)^n\Big) =  n\delta_\e + o(\delta_\e)$$
so that $\chi_{D_\e}\to 0$ in $L^1(Q)$.

We then compute the energy associated to $u_\e$ and $\chi_{D_\e}$:
\begin{align*}
\E_\e(u_\e,\chi_{D_\e}) &= \frac{\eta_\e}{2} \int_{D_\e} \mathbf{A}_w e(u_\e):e(u_\e)\, dx + \frac{\kappa}{\e} \LL^n(D_\e)\\
&\leq  C\eta_\e \sum_{i \in \{0,\ldots,N-1\}^n}   \frac{|u_i|^2}{\delta_\e^2} \LL^n(D_\e)+\frac{\kappa}{\e} \LL^n(D_\e) \\
&\leq C \left(\frac{\eta_\e}{\delta_\e} + \frac{\delta_\e}{\e}\right)\to 0
\end{align*}
where we used the fact that $\eta_\e/\delta_\e \to 0$ and $\delta_\e/\e \to 0$. As a consequence,
$$\E''_0(u,0) \leq \limsup_{\e \to 0} \E_\e(u_\e,\chi_{D_\e}) =0.$$

\medskip

\noindent {\it Step 2.} Next, if $u \in L^1(Q;\R^n)$ is arbitrary, then there exists a sequence $(u_N)_{N \in \N}$ as in~\eqref{eq:pwcst} such that $u_N \to u$ in $L^1(Q;\R^n)$. By the lower semicontinuity of the $\Gamma$-upper limit and the result of Step 1, we infer that
$$\E''_0(u,0) \leq \liminf_{N \to +\infty} \E''_0(u_N,0)=0,$$
completing the proof.
\end{proof}

\section{The elasticity regime}

\begin{thm}\label{thm:alpha=infty}
Let $\O \subset \R^n$ be a bounded open set and let $\mathbf{A}_w$, $\mathbf{A}_s$ be fourth-order symmetric elasticity tensors satisfying~\eqref{eq:growth}. If $\alpha=\infty$ in~\eqref{eq:alpha}, then the functionals $\E_\e$ $\Gamma$-converge as $\e \to 0$ with respect to the strong $L^1(\O;\R^n) \times L^1(\O)$-topology to the functional
$\Phi_\infty:L^1(\O;\R^n) \times L^1(\O) \to [0,+\infty]$ defined by
$$\Phi_\infty(u,\chi)=
\begin{cases}
\ds  \frac12 \int_\O \mathbf{A}_s e(u):e(u)\, dx & \text{ if }\chi=0\text{ a.e.\ and }u \in H^1(\O;\R^n),\\
+\infty & \text{ otherwise.}
\end{cases}$$
\end{thm}

\begin{proof}
The upper bound $\E''_0(u,\chi) \leq \Phi_\infty(u,\chi)$ is obvious if the right-hand side is infinite. If $\Phi_\infty(u,\chi)<\infty$, then $u \in H^1(\O;\R^n)$ and $\chi=0$, and choosing $u_\e := u$ and $\chi_\e := 0$ for all $\e > 0$, we get that
$$\E''_0(u,0) \leq \liminf_{\e \to 0} E_\e(u,0)= \frac12 \int_\O \mathbf{A}_s e(u):e(u)\, dx=\Phi_\infty(u,0),$$
as required. 

The remainder of the proof consists in establishing the lower bound. Clearly, $\E'_0(u,\chi) \geq \Phi_\infty(u,\chi)$ if the left-hand side is infinite, so that we can assume without loss of generality that  $\E'_0(u,\chi) <\infty$, and, by Lemma~\ref{lem:comp1}, that $\chi=0$ and $u \in BD(\O)$. We start by improving the compactness result in this particular regime by showing that, actually, $u \in H^1(\O;\R^n)$. To this aim, as in Lemma~\ref{lem:comp1}, let us consider a subsequence $\e_k \to 0^+$ and a sequence $(u_k,\chi_k)_{k \in \N} \subset H^1(\O;\R^n) \times L^\infty(\O;\{0,1\})$ such that $(u_k,\chi_k) \to (u,0)$ in $L^1(\O;\R^n) \times L^1(\O)$ and 
$$\lim_{k \to \infty} E_{\e_k}(u_k,\chi_k)=\E'_{0}(u,0)<+\infty.$$
According to the coercivity properties of the tensors $\mathbf{A}_w$ and $\mathbf{A}_s$, we have the following energy bound:
\begin{equation}\label{eq:lower-bound}
M_k := \frac12 \int_\O \big[\eta_{\e_k}  c_w \chi_k +  c_s(1-\chi_k)\big] \cdot |e(u_k)|^2\, dx + \frac{\kappa}{\e_k}\int_\O \chi_k\, dx
\leq M < \infty.
\end{equation}

\medskip

\noindent {\it Step 1: The one-dimensional case. }
By outer regularity of the Lebesgue measure, we can assume without loss of generality that the damaged set $D_k=\{\chi_k=1\}$ is open, and that it is actually a finite union of pairwise disjoint open intervals, i.e.,
$$D_k=\bigcup_{i=1}^{m_k} (a_i^k,b_i^k),$$
where $m_k \in \N$ and $a_i^k<b_i^k <a_{i+1}^k <b_{i+1}^k$ for all $1 \leq i \leq m_k-1$.  
We observe that minimizing the expression~\eqref{eq:lower-bound} with respect to all $\chi\in L^\infty(\O;\{0,1\})$, one finds that the minimizer is given by the characteristic function of the set 
$$\left\{ x \in \O : \; |u'_k(x)|>\sqrt{\frac{2\kappa}{( c_s - \eta_{\e_k} c_w )\e_k}}\right\},$$
which corresponds to the completely damaged part of the medium. It is therefore natural to expect the singularities to nucleate inside this set, and  the medium to remain elastic in the complementary set.

We then modify the function $u_k$ inside each interval $(a_i^k,b_i^k)$, where we distinguish two cases. Let us define the sets of indices 
$$I_k:=\left\{i \in \{1,\ldots,m_k\} : \; \frac{|u_k(b_i^k)-u_k(a_i^k)|}{b_i^k-a_i^k}>\sqrt{\frac{2\kappa}{( c_s - \eta_{\e_k} c_w )\e_k}}\right\}$$
and 
$$J_k:=\{1,\ldots,m_k\}\setminus I_k.$$
In the intervals $(a_i^k,b_i^k)$ where $i \in I_k$, it will be convenient to create a jump, while if $i \in J_k$, the values of $u_k(a_i^k)$ and $u_k(b_i^k)$ will be connected in an affine way. We therefore define
$$v_k(x):=\begin{cases}
u_k(x) & \text{ if } x \notin D_k,\\
u_k(a_i^k) & \text{ if } x \in \left(a_i^k,\frac{a_i^k+b_i^k}{2}\right) \text{ with }i \in I_k,\\
u_k(b_i^k) & \text{ if } x \in \left(\frac{a_i^k+b_i^k}{2},b_i^k\right) \text{ with }i \in I_k,\\
u_k(a_i^k) + (x-a_i^k)\frac{u_k(b_i^k)-u_k(a_i^k)}{b_i^k-a_i^k} & \text{ if }x \in (a_i^k,b_i^k) \text{ with }i \in J_k.
\end{cases}$$
Clearly, $v_k \in SBV(\O)$ with jump set $J_{v_k} = \bigcup_{i \in I_k}\left\{\frac{a_i^k+b_i^k}{2}\right\}$. We denote by $v_k'$ the approximately continuous part of the derivative $Dv_k$, for which we have $v_k' \in L^2(\O)$.

Let us compute each term of the energy. First,
\begin{align}
\frac{\kappa}{\e_k}\LL^1(D_k)&=\frac{\kappa}{\e_k}\sum_{i \in I_k}(b_i^k-a_i^k) + \frac{\kappa}{\e_k}\sum_{i \in J_k}(b_i^k-a_i^k) \notag\\
&\geq \frac{\kappa}{\e_k}\sum_{i \in I_k}(b_i^k-a_i^k) + \frac{ c_s- \eta_{\e_k} c_w}{2}\sum_{i \in J_k}(b_i^k-a_i^k)\frac{|u_k(b_i^k)-u_k(a_i^k)|^2}{(b_i^k-a_i^k)^2} \notag\\
&=\frac{\kappa}{\e_k}\sum_{i \in I_k}(b_i^k-a_i^k) + \frac{ c_s- \eta_{\e_k} c_w}{2}\sum_{i \in J_k}\int_{a_i^k}^{b_i^k}|v'_k|^2\, dx.\label{eq:diss}
\end{align}
Moreover, since $v_k=u_k$ in $\O \setminus D_k$, we get that
\begin{equation}\label{eq:beta}
\frac{ c_s}{2}\int_\O (1-\chi_k)|u'_k|^2\, dx = \frac{ c_s}{2}\int_{\O\setminus D_k}|v'_k|^2\, dx.
\end{equation}
Finally, owing to Jensen's inequality,
\begin{align}
\frac{ \eta_{\e_k} c_w}{2}\int_\O \chi_k|u'_k|^2\, dx
&= \frac{ \eta_{\e_k} c_w}{2}\sum_{i=1}^{m_k}\int_{a_i^k}^{b_i^k}|u'_k|^2\, dx \notag\\
&\geq  \frac{ \eta_{\e_k} c_w}{2}\sum_{i=1}^{m_k} (b_i^k-a_i^k)\frac{|u_k(b_i^k)-u_k(a_i^k)|^2}{(b_i^k-a_i^k)^2} \notag\\
&=  \frac{ \eta_{\e_k} c_w}{2}\sum_{i \in I_k} (b_i^k-a_i^k)\frac{|u_k(b_i^k)-u_k(a_i^k)|^2}{(b_i^k-a_i^k)^2} +\frac{ \eta_{\e_k} c_w}{2}\sum_{i \in J_k} \int_{a_i^k}^{b_i^k}|v'_k|^2\, dx. \label{eq:alpha1}
\end{align}
Gathering~\eqref{eq:lower-bound},~\eqref{eq:diss},~\eqref{eq:beta} and~\eqref{eq:alpha1} and using that $v'_k=0$ a.e.\ in $\bigcup_{i \in I_k} (a_i^k,b_i^k)$ yields
\begin{align*}
M_k &\geq \frac{ \eta_{\e_k} c_w}{2}\sum_{i \in I_k} (b_i^k-a_i^k)\frac{|u_k(b_i^k)-u_k(a_i^k)|^2}{(b_i^k-a_i^k)^2} +\frac{ \eta_{\e_k} c_w}{2}\sum_{i \in J_k} \int_{a_i^k}^{b_i^k}|v'_k|^2\, dx\\
&\qquad + \frac{ c_s}{2}\int_{\O\setminus D_k}|v'_k|^2\, dx + \frac{\kappa}{\e_k}\sum_{i \in I_k}(b_i^k-a_i^k) + \frac{ c_s- \eta_{\e_k} c_w}{2}\sum_{i \in J_k}\int_{a_i^k}^{b_i^k}|v'_k|^2\, dx \\
&= \frac{ c_s}{2} \int_\O|v'_k|^2\, dx+\sum_{i \in I_k}(b_i^k-a_i^k)\left[\frac{ \eta_{\e_k} c_w}{2} \cdot \frac{|u_k(b_i^k)-u_k(a_i^k)|^2}{(b_i^k-a_i^k)^2}+\frac{\kappa}{\e_k} \right].
\end{align*}
Thanks to Young's inequality we deduce that
\begin{align*}
M \geq M_k &\geq \frac{ c_s}{2} \int_\O|v'_k|^2\, dx+\sqrt{\frac{2\kappa  c_w\eta_{\e_k}}{\e_k}}\sum_{i \in I_k} |u_k(b_i^k)-u_k(a_i^k)|\\
&= \frac{ c_s}{2} \int_\O|v'_k|^2\, dx+\sqrt{\frac{2\kappa  c_w\eta_{\e_k}}{\e_k}}\int_{J_{v_k}}|v_k^+-v_k^-|\, d\HH^0.
\end{align*}

The previous formula implies that the sequence $(v_k)_{k\in\N}$ is uniformly bounded in $BV(\O)$, and thus a subsequence converges weakly* in $BV(\O)$ to some $v \in BV(\O)$. In addition, since $\{ u_k \neq v_k\} \subset D_k$ and $\LL^1(D_k) \to 0$ by~\eqref{eq:lower-bound}, we infer that $u \in BV(\O)$ and that the whole sequence $(v_k)$ converges weakly* to $u$. Since $(v_k')_{k \in \N}$ is bounded in $L^2(\O)$ and $|D^s v_k|(\O) \to 0$ (since $\frac{\eta_{\e_k}}{\e_k} \to \infty$), we actually deduce that $u \in H^1(\O)$. Passing to the lower limit in the previous formula thus yields
\begin{equation}\label{eq:E_k}
\liminf_{k \to \infty} M_k\geq \frac{ c_s}{2} \int_\O|u'|^2\, dx.
\end{equation}
Moreover, since $v_k=u_k$ a.e.\ in $\O \setminus D_k$, $v'_k \wto u'$ weakly in $L^2(\O)$ and $\chi_k \to 0$ strongly in $L^2(\O)$, we also get that
\begin{equation}\label{eq:lsc}
\liminf_{k \to \infty}  \int_\O (1-\chi_k)|u'_k|^2\, dx\geq \int_\O|u'|^2\, dx.
\end{equation}

\medskip

\noindent {\it Step 2: The $n$-dimensional case.} The general case will be deduced from the one-dimensional case via standard slicing techniques. 

We start by introducing some notation. For $\nu \in \Sn^{n-1}$, we denote by $\Pi_\nu$ the hyperplane orthogonal to $\nu$ and passing through the origin. Given a set $E \subset \R^n$, a scalar function $g:E \to \R$, and a vector map $f:E \to \R^n$, for all $y \in \Pi_\xi$, we denote by
$$E^\nu_y:=\bigl\{t \in \R : y+t\nu \in E\bigr\}, \quad g_y^\nu(t):=g(y+t\nu), \quad f_y^\nu(t):=f(y+t\nu)\cdot\nu \;\;\text{ for }t \in E_y^\nu$$
the sections of $E$, $g$ and $f$, respectively, that pass through $y \in \Pi_\nu$ in the direction $\nu$. 

Using Fubini's theorem, for all $\nu \in \Sn^{n-1}$, there exists a subsequence (possibly depending on $\nu$),  denoted by $(u_j,\chi_j)=(u_{k_j},\chi_{k_j})$, such that 
$$\liminf_{k \to +\infty} M_k=\lim_{j \to +\infty}M_{k_j}$$
and
\begin{equation}\label{eq:conv-slicing}
\bigl( (u_j)_y^\nu,(\chi_j)_y^\nu \bigr) \to \bigl( u_y^\nu,0 \bigr) \text{ in }L^1(\O_y^\nu;\R^d) \times L^1(\O_y^\nu)\quad\text{ for $\HH^{n-1}$-a.e.\ $y \in \Pi_\nu$.}
\end{equation}
Using the structure theorem in $BD$ (see~\cite[Theorem 4.5]{ACDM}) and the fact that for $\HH^{n-1}$-a.e.\ $y \in \Pi_\nu$ we have 
$$|((u_j)_y^\nu)'(t)| = |e(u_j)(y+t\nu):(\nu\otimes \nu)|\leq  |e(u_j)(y+t\nu)|\quad \LL^1\text{-a.e.\ in }\O_y^\nu,$$
Fatou's lemma leads to
\begin{multline}\label{eq:Fatou}
M \geq  \int_{\Pi_\nu} \liminf_{j \to +\infty} \bigg\{\int_{\O_y^\nu} \Big[ \frac12 \left( c_w \eta_{\e_{k_j}} (\chi_j)_y^\nu(t) +  c_s (1-(\chi_j)_y^\nu(t))\right)|((u_j)_y^\nu)'(t)|^2\\
+\frac{\kappa}{\e_{k_j}} (\chi_j)_y^\nu(t)\Big] \, dt \bigg\}\, d\HH^{n-1}(y).
\end{multline}
Thanks to the result in the one-dimensional case, in particular~\eqref{eq:E_k}, and ~\eqref{eq:conv-slicing}, we get that $u_y^\nu \in H^1(\O_y^\nu)$ for $\HH^{n-1}$-a.e.\ $y \in \Pi_\nu$ (in particular $D^s u_y^\nu = 0$), and
\begin{multline}\label{eq:1Dcase}
\liminf_{j \to +\infty} \int_{\O_y^\nu} \Big[ \frac12 \left( c_w \eta_{\e_{k_j}} (\chi_j)_y^\nu(t) +  c_s (1-(\chi_j)_y^\nu(t))\right)|((u_j)_y^\nu)'(t)|^2+\frac{\kappa}{\e_{k_j}} (\chi_j)_y^\nu(t)\Big] \, dt\\
\geq\frac{ c_s}{2} \int_{\O_y^\nu}|(u_y^\nu)'(t)|^2\, dt.
\end{multline}
Integrating~\eqref{eq:1Dcase} with respect to $y \in \Pi_\nu$ and using~\eqref{eq:Fatou}  gives
$$\frac{ c_s}{2}\int_{\Pi_\nu} \int_{\O_y^\nu}|(u_y^\nu)'(t)|^2\, dt \, d\HH^{n-1}(y)\leq M.$$

According to the structure theorem in $BD$ (see~\cite[Theorem 4.5]{ACDM}) we have
$$\begin{cases}
(u_y^\nu)'(t)=e(u)(y+t\nu):(\nu\otimes\nu)\text{  for $\HH^{n-1}$-a.e.\ $y \in \Pi_\nu$ and for $\LL^1$-a.e.\ $t \in \O_y^\nu$},\\
|E^s u:( \nu\otimes\nu)|(\O)=\int_{\Pi_\nu} |D^s u_y^\nu|(\O_y^\nu)\, d\HH^{n-1}(y).
\end{cases}$$
Therefore, Fubini's theorem yields for all $\nu \in \Sn^{n-1}$,
$$\int_\O |e(u):(\nu\otimes\nu)|^2\, dx <+\infty, \quad |E^s u:( \nu\otimes\nu)|(\O)=0.$$
Choosing first $\nu=e_i$ and then $\nu=(e_i+e_j)/2$ for all $1 \leq i,j\leq n$, where $\{e_1,\ldots,e_n\}$ stands for the canonical basis of $\R^n$, implies that $e(u) \in L^2(\O;\Ms)$ and $|E^s u|(\O)=0$ which means that $u \in H^1(\O;\R^n)$.

\medskip

\noindent {\it Step 3:  Weak convergence of the strain.} According to~\eqref{eq:lsc} and Fatou's lemma, the previous argument also shows that
$$\liminf_{k \to +\infty}\int_\O (1-\chi_k)|e(u_k):(\nu\otimes \nu)|^2\, dx \geq \int_\O |e(u):(\nu\otimes\nu)|^2\, dx.$$
We can further use the same method to establish that for all $w \in L^2(\O)$,
\begin{equation}\label{eq:w}
\liminf_{k \to +\infty}\int_\O (1-\chi_k)|e(u_k):(\nu \otimes \nu)-w|^2\, dx \geq \int_\O |e(u):(\nu\otimes\nu)-w|^2\, dx.
\end{equation}
Indeed, the previous inequality clearly holds if $w$ is piecewise constant on a Lipschitz partition of $\O$, and the general case follows from a density argument.

Since the sequence $((1-\chi_k)e(u_k))_{k \in \N}$ is bounded in $L^2(\O;\Ms)$, we can extract a subsequence (not relabeled) and find some $A \in L^2(\O;\Ms)$ such that $(1-\chi_k)e(u_k) \wto A$ weakly in $L^2(\O;\Ms)$. Applying~\eqref{eq:w} with $w:=A:(\nu\otimes\nu)-tz$, where $t \in \R$ and $z \in L^2(\O)$, we infer that
\begin{align*}
&\int_\O |(e(u)-A):(\nu\otimes\nu)|^2\, dx + 2t\int_\O z \cdot (e(u)-A):(\nu\otimes\nu)\,dx\\
&\qquad\leq \liminf_{k \to \infty}\int_\O (1-\chi_k)|(e(u_k)-A):(\nu \otimes \nu)|^2\, dx,
 \end{align*}
where we used that $(1-\chi_k)e(u_k) \wto A$ weakly in $L^2(\O;\Ms)$ and $\chi_k \to 0$ strongly in $L^2(\O)$. Passing to the limit as $t \to \pm\infty$ yields
$$\int_\O z(e(u)-A):(\nu\otimes\nu)\,dx=0$$
for all $\nu \in \Sn^{n-1}$ and all $z \in L^2(\O)$, which implies that $A=e(u)$ a.e.\ in $\O$. By uniqueness of the weak limit, we infer that also for the full sequence $(1-\chi_k)e(u_k) \wto e(u)$ weakly in $L^2(\O;\Ms)$. Finally, since
$$\E'_0(u,\chi)=\lim_{k \to \infty}\Phi_{\e_k}(u_k,\chi_k) \geq \liminf_{k \to \infty} \frac12 \int_\O (1-\chi_k) \mathbf{A}_s e(u_k):e(u_k)\, dx,$$
we deduce that 
$$\E'_0(u,\chi)\geq  \frac12 \int_\O \mathbf{A}_s e(u):e(u)\, dx=\Phi_\infty(u,0),$$
which completes the proof of the lower bound.
\end{proof}

\section{The Tresca model}

In this section we consider a different scaling of the energy. The weak elastic tensor $\eta_\e\mathbf{A}_w$ will be replaced by a new tensor $\mathbf{A}_w^\e$, in which the small parameter $\eta_\e$ will not act on the divergence term. For reasons of notational simplicity, we only consider the case $\eta_\e = \e$ here.
We assume that $\mathbf{A}^\e_w$ and $\mathbf{A}_s$ are isotropic tensors, i.e., for all $\xi \in \Ms$,
\begin{align*}
\mathbf{A}_w^\e \xi & := \lambda_w (\tr\xi)\, \Id+ 2\e\mu_w \xi,\\
\mathbf{A}_s \xi & := \lambda_s (\tr\xi)\, \Id+ 2\mu_s \xi,
\end{align*}
where $\lambda_i>0$ and $\mu_i>0$ are the Lam\'e coefficients, which satisfy $\lambda_w \leq \lambda_s$. For every $u \in H^1(\O;\R^n)$, $\chi \in L^\infty(\O;\{0,1\})$, and any $\e>0$, we define the following brittle damage energy functional:
\[\label{Etilde}
\widetilde E_\e(u,\chi):=\frac12 \int_\O \bigl[\chi \mathbf{A}_w^\e + (1-\chi)\mathbf{A}_s \bigr] e(u):e(u)\, dx + \frac{\kappa}{\e} \int_\O \chi\, dx.
\]
We will show that the limit model remains of plasticity type but with a Tresca elasticity set 
$$\widetilde K:=\bigl\{\tau \in \Msd : \tau_n-\tau_1 \leq 2 \sqrt{2\kappa \mu_w}\bigr\},$$
where $\tau_1 \leq \cdots \leq \tau_n$ are the ordered eigenvalues of $\tau$. Contrary to the model obtained in Theorem~\ref{thm:alpha=1}, here the stress constraint relates only to the deviatoric part of the stress.

It is convenient to introduce the Temam--Strang space~\cite{Temam}
\[
  U(\O) := \bigl\{ u \in BD(\O) :\;  \dive u \in L^2(\O) \bigr\},
\]
that is, the space of $BD$ functions whose distributional divergence is absolutely continuous with respect to Lebesgue measure and possesses a square-integrable density. This implies in particular that $E^su=E^s_Du$, the deviatoric part of $Eu$. The space $U(\Omega)$ is a Banach space under the norm
$$\|u\|_{U(\Omega)} :=\|u\|_{BD(\Omega)} +\|\dive u\|_{L^2(\Omega)}.$$
}
The main result of the section is the following.

\begin{thm}\label{thm:tilde}
Let $\O \subset \R^n$ ($n=2$ or $n=3$) be a bounded open set with Lipschitz boundary. For every $\e>0$ define the functional $\widetilde \E_\e:L^1(\O;\R^n) \times L^1(\O) \to [0,+\infty]$ by
\[
\widetilde \E_\e(u,\chi):=
\begin{cases}
\widetilde E_\e(u,\chi) & \text{ if } (u,\chi) \in H^1(\O;\R^n) \times L^\infty(\O;\{0,1\}),\\
+\infty & \text{ otherwise.}
\end{cases}
\]
Then, the functionals $\widetilde \E_\e$ $\Gamma$-converge as $\e \to 0$ with respect to the strong $L^1(\O;\R^n) \times L^1(\O)$-topology to the functional
$\widetilde \E_0:L^1(\O;\R^n) \times L^1(\O) \to [0,+\infty]$ defined by
\[
\widetilde \E_0(u,\chi):=
\begin{cases}
\begin{aligned}
\ds    &\int_\O\left(\frac{\lambda_s}{2}+\frac{\mu_s}{n}\right)(\dive u)^2 \, dx\\
&\qquad \ds + \int_\O \widetilde W(e_D(u))\, dx + \int_\O  \sqrt{2\kappa \tilde h\left(\frac{dE^s_D u}{d|E^s_D u|} \right)} \,d|E^s_D u|
\end{aligned}
&\text{ if }
\begin{cases}
\chi=0\text{ a.e.},\\
u \in U(\O),
\end{cases}\\
+\infty & \text{ otherwise,}
\end{cases}
\]
where
\begin{equation}\label{eq:tildefh}
\tilde f(\xi):=\mu_s|\xi|^2, \qquad \tilde h(\xi):=\mu_w \left(\sum_{i=1}^n|\xi_i| \right)^2 \quad \text{ for all } \xi \in \Msd,
\end{equation}
with $\xi_1 \leq \cdots \leq \xi_n$ the ordered eigenvalues of $\xi$, and
$\widetilde W$ is defined on $\Msd$ via
$$\widetilde W:=\tilde f \, \Box \, \sqrt{2\kappa \tilde h}.$$
\end{thm}

For all $\xi \in \Ms$, let
\[
\label{d:tildefg}
 \widetilde W_\e(\xi):=\min\left\{\frac12 \mathbf{A}_s \xi:\xi,\frac12 \mathbf{A}_w^\e \xi:\xi + \frac{\kappa}{\e}\right\}.
\]
Denoting by $SQ\widetilde W_\e$ the symmetric quasiconvex envelope of $\widetilde W_\e$, from~\cite[Proposition 5.2]{AL} we know that it can be expressed as
$$SQ \widetilde W_\e(\xi)=\min_{0 \leq \theta \leq 1} \widetilde F_\e(\theta,\xi),$$
where 
\begin{multline*}
\widetilde F_\e(\theta,\xi) :=\frac12 \mathbf{A}_w^\e\xi:\xi + \frac{\kappa\theta}{\e}+ (1-\theta) \max_{\tau \in \Ms} \left\{\tau:\xi -\frac12 (\mathbf{A}_s-\mathbf{A}_w^\e)^{-1}\tau:\tau -\frac{\theta}{2\e}\widetilde G_\e(\tau)\right\}\\
=\frac12 \mathbf{A}_w^\e\xi:\xi + \frac{\kappa\theta^2}{\e}+ (1-\theta) \max_{\tau \in \Ms} \left\{\tau:\xi -\frac12 (\mathbf{A}_s-\mathbf{A}_w^\e)^{-1}\tau:\tau +\frac{\theta}{2\e}\left(2\kappa- \widetilde G_\e(\tau)\right)
\right\}
\end{multline*}
and, if $\tau_1 \leq \cdots \leq \tau_n$ are the ordered eigenvalues of $\tau \in \Ms$, 
$$\widetilde G_\e(\tau):=
\begin{cases}
\frac{\tau_1^2}{2\lambda_w/\e+\mu_w}& \text{ if } \frac{\lambda_w+2\e\mu_w}{2(\lambda_w+\e\mu_w)}(\tau_1+\tau_n) < \tau_1,\\
\frac{(\tau_1-\tau_n)^2}{4\mu_w}+\frac{(\tau_1+\tau_n)^2}{4(\lambda_w/\e+\mu_w)} & \text{ if }  \tau_1 \leq \frac{\lambda_w+2\e\mu_w}{2(\lambda_w+\e\mu_w)}(\tau_1+\tau_n)\leq \tau_n,\\
\frac{\tau_n^2}{2\lambda_w/\e+\mu_w}& \text{ if } \tau_n < \frac{\lambda_w+2\e\mu_w}{2(\lambda_w+\e\mu_w)}(\tau_1+\tau_n).
\end{cases}
$$
Let us  also denote by
$$\widetilde G(\tau):=\frac{(\tau_1-\tau_n)^2}{4\mu_w}$$
the pointwise limit of $\widetilde G_\e(\tau)$ as $\e \to 0$, which in particular satisfies $\widetilde G(\tau)=\widetilde G(\tau_D)$, where $\tau_D$ denotes the deviatoric part of $\tau$.

We first compute the pointwise limit of the family $(SQ\widetilde W_\e)_{\e>0}$ in order to get a candidate for the effective bulk energy density. 
\begin{prop}
For all $\xi \in \Ms$, we have
$$SQ\widetilde W_\e(\xi) \to (\tr\xi)^2\left(\frac{\lambda_s}{2}+\frac{\mu_s}{n}\right)+\widetilde W(\xi_D),$$
where
$$\widetilde W:=(\tilde f^*+I_{\widetilde K})^* \quad \text{ in }\Msd$$
where $\widetilde K:=\bigl\{\tau \in \Msd:  \; \widetilde G(\tau) \leq 2\kappa\bigr\}$ is the Tresca elasticity set, $\tilde f$ is defined in~\eqref{eq:tildefh}, and the conjugations are to be understood in $\Msd$.
\end{prop}

\begin{proof}
	Fix $\xi \in \Ms$. We will prove that $(\widetilde F_\e(\cdot,\xi))_{\e>0}$ $\Gamma$-converges in $[0,1]$ to the function $\widetilde F_0(\cdot,\xi)$ defined by $\widetilde F_0(\theta,\xi):=(\tr\xi)^2\left(\frac{\lambda_s}{2}+\frac{\mu_s}{n}\right)+\widetilde W(\xi_D)$ if $\theta=0$ and $\widetilde F_0(\theta,\xi):=+\infty$ if $\theta\neq 0$.
	
\medskip
	
\noindent {\it Lower bound:} Let $(\theta_\e)_{\e>0}$ be a sequence in $[0,1]$. If $\liminf_\e \widetilde F_\e(\theta_\e,\xi)=+\infty$, there is nothing to prove. Without loss of generality, we can therefore assume that $\liminf_\e \widetilde F_\e(\theta_\e,\xi)<+\infty$. Moreover, up to a subsequence we can also suppose that the previous lower limit is actually a limit, and that $\theta_\e\to \theta \in [0,1]$. Since $\widetilde F_\e(\theta_\e,\xi) \geq \frac{\kappa\theta_\e}{\e}$ (choose $\tau = 0$), we deduce that $\theta=0$. We next estimate $\widetilde F_\e$ from below as follows: for all $\tau \in \Ms$,
\begin{equation}\label{eq:taueps}
\widetilde F_\e(\theta_\e,\xi)\geq \frac{\lambda_w}{2}(\tr\xi)^2+ (1-\theta_\e) \left\{ \tau:\xi  -\frac12 (\mathbf{A}_s-\mathbf{A}_w^\e)^{-1}\tau:\tau  +\frac{\theta_\e}{2\e} \left({2\kappa}-\widetilde G_\e(\tau)\right) \right\}.
\end{equation}
We claim that for all $\tau \in \Ms$ with $\tau_D \in \widetilde K$ and for all $\e>0$ small enough there exists $\tau_\e \in \Ms$ such that $\widetilde G_\e(\tau_\e) \leq 2\kappa$ and $\tau_\e \to \tau$. Indeed, on the one hand, if $(\tau_D)_1<(\tau_D)_n$, since $(\tau_D)_i=\tau_i-\frac{1}{n} \tr\tau$, we deduce that $\tau_1<\tau_n$. Thus, for $\e$ small we have
$$ \tau_1 \leq \frac{\lambda_w+2\e\mu_w}{2(\lambda_w+\e\mu_w)}(\tau_1+\tau_n)\leq \tau_n$$
and 
$$\widetilde G_\e(\tau)=\frac{(\tau_1-\tau_n)^2}{4\mu_w}+\frac{(\tau_1+\tau_n)^2}{4(\lambda_w/\e+\mu_w)} >0.$$
Setting $$\tau_\e:=\sqrt{\frac{\widetilde G(\tau)}{\widetilde G_\e(\tau)}} \tau,$$ 
we deduce that $\tau_\e \to \tau$ since $\widetilde G_\e(\tau)\to \widetilde G(\tau)$. In addition, using the $2$-homogeneity of $\widetilde G_\e$, we also have $\widetilde G_\e(\tau_\e) = \widetilde G(\tau) \leq 2\kappa$.
	
On the other hand, if $\tau_1=\tau_n$, then $\widetilde G_\e(\tau)\to 0$ as $\e\to0$ and in particular $\widetilde G_\e(\tau_\e)\leq 2\kappa$ for $\tau_\e:=\tau$ for $\e > 0$ sufficiently small. Writing~\eqref{eq:taueps} with $\tau_\e$, and passing to the limit as $\e \to 0$ we deduce that
$$\liminf_{\e \to 0}\widetilde F_\e(\theta_\e,\xi) \geq \frac{\lambda_w}{2}(\tr\xi)^2+\frac{(\tr\tau)(\tr\xi)}{n}+\tau_D:\xi_D-\frac{(\tr\tau)^2}{2n(n(\lambda_s-\lambda_w)+2\mu_s)}-\frac{1}{4\mu_s}|\tau_D|^2.$$
Here we used that for all $\tau \in \Ms$, $\e > 0$,
\begin{equation} \label{eq:CsCw_inv}
(\mathbf{A}_s-\mathbf{A}_w^\e)^{-1}\tau = \frac{\tr \tau}{n(n(\lambda_s-\lambda_w)+2(\mu_s-\e\mu_w))} \Id + \frac{1}{2(\mu_s-\e\mu_w)} \tau_D,
\end{equation}
which follows from a straightforward computation.
Maximizing first with respect to $\tr\tau\in \R$ and then with respect to $\tau_D\in \widetilde K$, we obtain
\begin{align*}
\liminf_{\e \to 0}\widetilde F_\e(\theta_\e,\xi) & \geq	(\tr\xi)^2\left(\frac{\lambda_s}{2}+\frac{\mu_s}{n}\right)+
 \sup_{\tau_D \in \widetilde K} \left\{\tau_D:\xi_D -\frac1{4\mu_s}|\tau_D|^2\right\}\\
 & = (\tr\xi)^2\left(\frac{\lambda_s}{2}+\frac{\mu_s}{n}\right)+(\tilde{f}^*+I_{\widetilde K})^*(\xi_D)\\
 &= (\tr\xi)^2\left(\frac{\lambda_s}{2}+\frac{\mu_s}{n}\right)+\widetilde W(\xi_D).
 \end{align*}
	
\medskip 
	
\noindent {\it Upper bound:} If $\theta \neq 0$, there is nothing to prove. We can thus assume without loss of generality that $\theta=0$. Let $\lambda\geq 0$ and set $\theta_\e:=\lambda\e \to 0$. Then, using~\eqref{eq:CsCw_inv} again,
\begin{align*}
\widetilde F_\e(\theta_\e,\xi) &= \frac12\mathbf{A}_w^\e\xi:\xi + \kappa\lambda^2\e \\
&\qquad+ (1-\lambda\e) \sup_{\tau \in \Ms} \biggl\{\frac{( \tr\tau)(\tr\xi)}{n}+\tau_D:\xi_D -\frac{(\tr\tau)^2}{2n\big(n(\lambda_s-\lambda_w)+2(\mu_s-\e\mu_w)\big)} \\
&\qquad\hspace{88pt} -\frac{1}{4(\mu_s -\e\mu_w)}|\tau_D|^2
+ \frac{\lambda}{2} \big(2\kappa -\widetilde G_\e(\tau)\big) \biggr\}.
\end{align*}
Notice that, since the supremum in the previous expression is nonnegative for every $\e$, it is in fact obtained on a compact subset of $\Ms$, which is independent of $\e$, as it can be easily checked. Thus, we may pass to the limit as $\e\to 0$ and then take the infimum in $\lambda\geq0$ to obtain (using~\cite[Chapter VI, Proposition 2.3]{ET} as in the proof of Proposition~\ref{prop:Walpha})
\begin{align*}
&\limsup_{\e \to 0} \widetilde F_\e(\theta_\e,\xi) - \frac{\lambda_w}{2}(\tr\xi)^2\\
&\qquad \leq \inf_{\lambda\geq 0}\sup_{\tau \in \Ms}  \left\{ \frac{(\tr\tau)(\tr\xi)}{n}+\tau_D:\xi_D-\frac{(\tr\tau)^2}{2n\big(n(\lambda_s-\lambda_w)+2\mu_s\big)}-\frac{1}{4\mu_s}|\tau_D|^2
+ \frac{\lambda}{2} \big(2\kappa -\widetilde G(\tau)\big) \right\}\\
&\qquad =  \sup_{\tau \in \Ms} \inf_{\lambda\geq 0}  \left\{ \frac{(\tr\tau)(\tr\xi)}{n}+\tau_D:\xi_D-\frac{(\tr\tau)^2}{2n(n(\lambda_s-\lambda_w)+2\mu_s)}-\frac{1}{4\mu_s}|\tau_D|^2
+ \frac{\lambda}{2} (2\kappa -\widetilde G(\tau)) \right\} \\
&\qquad =\sup_{T \in \R} \left\{\frac{\tr\xi}{n}T-\frac{T^2}{2n(n(\lambda_s-\lambda_w)+2\mu_s)} \right\} + \sup_{\tau_D \in \widetilde K} \left\{\tau_D:\xi_D-\frac{1}{4\mu_s}|\tau_D|^2 \right\},
\end{align*}
from which we deduce that
$$\limsup_{\e \to 0} \widetilde F_\e(\theta_\e,\xi) \leq (\tr\xi)^2\left(\frac{\lambda_s}{2}+\frac{\mu_s}{n}\right)+(f^*+I_{\widetilde K})^*(\xi_D)=(\tr\xi)^2\left(\frac{\lambda_s}{2}+\frac{\mu_s}{n}\right)+\widetilde W(\xi_D).$$

\medskip
	
\noindent {\it Convergence of minimizers.} According to the fundamental theorem of $\Gamma$-convergence, we deduce that
$$SQ\widetilde W_\e(\xi)=\min_{0 \leq \theta \leq 1}\widetilde F_\e(\theta,\xi) \to \min_{0 \leq \theta \leq 1}\widetilde F_0(\theta,\xi)=(\tr\xi)^2\left(\frac{\lambda_s}{2}+\frac{\mu_s}{n}\right)+\widetilde W(\xi_D),$$
which completes the proof of the proposition.
\end{proof}

We next identify the support function of the Tresca elasticity set $\widetilde K$.

\begin{lem}\label{l:h2}
For all $\xi \in \Msd$,
$$I_{\widetilde K}^*(\xi)=\sqrt{2\kappa \tilde h(\xi)},$$	where $\tilde h$ is defined in~\eqref{eq:tildefh} and the conjugation is to be understood in $\Msd$. In particular, $\widetilde W=\tilde f\, \Box \, \sqrt{2\kappa \tilde h}$, where the inf-convolution is to be understood in $\Msd$.
\end{lem}

\begin{proof}
Arguing as in the proof of Lemma~\ref{l:h}, we only need to check that $\widetilde G^*=\tilde h/4$ in $\Msd$. For all $\lambda\geq 0$ and all $\tau \in \Ms$, let
$$G_\lambda(\tau):=
\begin{cases}
\frac{\tau_1^2}{\lambda+2\mu_w}& \text{ if } \frac{\lambda+2\mu_w}{2(\lambda+\mu_w)}(\tau_1+\tau_n)<\tau_1,\\
\frac{(\tau_1-\tau_n)^2}{4\mu_w}+\frac{(\tau_1+\tau_n)^2}{4(\lambda+\mu_w)} & \text{ if }  \tau_1 \leq \frac{\lambda+2\mu_w}{2(\lambda+\mu_w)}(\tau_1+\tau_n)\leq \tau_n,\\
\frac{\tau_n^2}{\lambda+2\mu_w}& \text{ if } \tau_n < \frac{\lambda+2\mu_w}{2(\lambda+\mu_w)}(\tau_1+\tau_n),
\end{cases}$$
and for all $\xi \in \Ms$,
$$h_\lambda(\xi):=\mu_w \left(\sum_{i=1}^n|\xi_i| \right)^2+(\lambda+\mu_w) \left(\sum_{i=1}^n \xi_i \right)^2.$$
Clearly, $h_\lambda(\xi) = \tilde h(\xi)$ for any $\lambda \geq 0$ if $\xi \in \Msd$. Thus, arguing as in the proof of Lemma~\ref{l:h}, we have for all $\xi \in \Msd$,
$$\frac{\tilde h(\xi)}{4}=\frac{h_\lambda(\xi)}{4}=\sup_{\tau\in\Ms} \bigl\{\tau:\xi-G_\lambda(\tau) \bigr\},$$
that is, the convex conjugate of $G_\lambda$ in the full space $\Ms$. We compute
\begin{align*}
\frac{\tilde h(\xi)}{4} &=\sup_{\lambda\geq0}\sup_{\tau \in \Ms}\bigl\{\tau:\xi-G_\lambda(\tau)\bigr\}=\sup_{\tau \in \Ms}\sup_{\lambda\geq0}\,\bigl\{\tau:\xi-G_\lambda(\tau)\bigr\}\\
&=\sup_{\tau \in \Ms}\bigl\{\tau:\xi-\inf_{\lambda\geq 0} G_\lambda(\tau)\bigr\}=\sup_{\tau \in \Ms}\bigl\{\tau:\xi-\widetilde G(\tau)\bigr\}\\
&=\sup_{\tau \in \Msd}\bigl\{\tau:\xi-\widetilde G(\tau)\bigr\}=\widetilde G^*(\xi),
\end{align*}
which concludes the proof.
\end{proof}
The following result is the analogue of Proposition~\ref{prop:W_alpha} in the present Tresca regime. The proof is identical, therefore it will be omitted.
\begin{prop}\label{prop:tildeW}
The function $\widetilde W$ is convex,
\[
c|\xi| - \frac{1}{c}\leq \widetilde W(\xi) \leq C|\xi| \quad \text{ for all }\xi \in \Msd,
\]
for some $c,C>0$, and 
\[
|\widetilde W(\xi_1) - \widetilde W(\xi_2)|\leq L|\xi_1-\xi_2| \quad \text{  for all }\xi_1, \; \xi_2 \in \Msd,
\]
for some $L>0$.
In addition, its recession function, defined for all $\xi \in \Msd$ by
$$\widetilde W^\infty(\xi)=\lim_{t \to +\infty}\frac{\widetilde W(t\xi)}{t},$$
exists and is given by
$$\widetilde W^\infty(\xi)=\sqrt{2\kappa \tilde h(\xi)} \quad \text{ for all }\xi \in \Msd.$$
Finally, for all $a,b \in \R^n$ with $a \cdot b=0$,
$$\widetilde W^\infty(a \odot b)=2\sqrt{\kappa \mu_w }|a \odot b|.$$
\end{prop}

We are now in position to prove Theorem~\ref{thm:tilde}. 

\begin{proof}[Proof of Theorem~\ref{thm:tilde}] 
	
\noindent {\it Step 1: The upper bound.}
An analogous argument to that used in the proof of Theorem~\ref{thm:alpha=1} (employing~\cite[Remark II.3.4]{Temam} and~\cite[Theorem 1.1]{KR} in place of~\cite[Proposition I.1.3]{Temam} and ~\cite[Corollary 1.10]{ARDPR}) shows that it is enough to establish the upper bound for $u \in W^{1,\infty}(\O;\R^n)$ and $\chi=0$. According to the dominated convergence theorem, we infer that
$$\int_\O\left(\frac{\lambda_s}{2}+\frac{\mu_s}{n}\right)(\dive u)^2 \, dx+ \int_\O \widetilde W(e_D(u))\, dx = \lim_{\e \to 0}\int_\O SQ\widetilde W_\e(e(u))\, dx.$$
For every $\e>0$,
$$v \in W^{1,1}(\O;\R^n) \mapsto \int_\O SQ\widetilde W_\e(e(v))\, dx$$
is the $L^1(\O;\R^n)$-lower semicontinuous envelope restricted to $W^{1,1}(\O;\R^n)$ of
$$v \in W^{1,1}(\O;\R^n) \mapsto \int_\O \widetilde W_\e(e(v))\, dx,$$
see~\cite{BFT,ARDPR}. It is thus possible to find a recovery sequence $(u^\e_k)_{k \in \N} \subset W^{1,1}(\O;\R^n)$ such that
$u^\e_k \to u$ in $L^1(\O;\R^n)$ as $k \to \infty$, and
$$\int_\O SQ\widetilde W_\e(e(u))\, dx=\lim_{k \to +\infty} \int_\O \widetilde W_\e(e(u^\e_k))\, dx.$$
Using a diagonalization argument, we extract a subsequence $k(\e) \to \infty$ as $\e \to 0$ such that $v_\e:=u_{k(\e)}^\e \to u$ in $L^1(\O;\R^n)$ and
$$\int_\O\left(\frac{\lambda_s}{2}+\frac{\mu_s}{n}\right)(\dive u)^2 \, dx+ \int_\O \widetilde W(e_D(u))\, dx=\lim_{\e \to 0} \int_\O \widetilde W_\e(e(v_\e))\, dx.$$
Then, defining the damaged sets as
$$D_\e:=\left\{x \in \O : \; (\mathbf{A}_s-\mathbf{A}_w^\e)e(v_\e)(x):e(v_\e)(x) \geq \frac{2\kappa}{\e}\right\},$$
we obtain by construction that
$$\limsup_{\e \to 0} \widetilde E_\e(v_\e,\chi_{D_\e})= \lim_{\e \to 0}\int_\O \widetilde W^\e(e(v_\e))\, dx= \widetilde \E_0(u,0),$$
which completes the proof of the upper bound.
	
\medskip
	
\noindent {\it Step 2: The lower bound.} 
For all $\xi \in \Msd$ we define
$$\tilde g_\e(\xi):=\e \mu_w|\xi|^2 +\frac{\kappa}{\e}.$$
Let $(u_\e,\chi_\e)_{\e>0}$ be a sequence in $L^1(\O;\R^n) \times L^1(\O)$ such that $u_\e \to u$ in $L^1(\O;\R^n)$, $\chi_\e \to 0$ in $L^1(\O)$ and $\liminf_\e \widetilde E_\e(u_\e,\chi_\e)<+\infty$. Up to a subsequence, we additionally have that $u_\e\rightharpoonup u$ weakly* in $BD(\O)$ and $\dive u_\e \rightharpoonup \dive u$ weakly in $L^2(\O)$. Moreover, since $\chi_\e \to 0$ strongly in $L^2(\O)$ and the sequence $((1-\chi_\e)\dive u_\e)_{\e>0}$ is bounded in $L^2(\O)$, we have that $(1-\chi_\e)\dive u_\e \rightharpoonup\dive u$ weakly in $L^2(\O)$. Using that $e(u_\e)=\frac{1}{n}(\dive u_\e){\rm Id}+e(u_\e)_D$ and the weak lower semicontinuity of the norm, we have
\begin{align*}
 &\liminf_{\e\to 0}\widetilde E_\e(u_\e,\chi_\e)\\
  &\qquad \geq  \liminf_{\e\to 0}\left\{\left(\frac{\lambda_s}{2}+\frac{\mu_s}{n}\right)\int_\O(1-\chi_\e)(\dive u_\e)^2 \, dx+\int_\O (1-\chi_\e) \tilde f(e_D(u_\e))+\chi_\e \tilde g_\e(e_D(u_\e))\, dx\right\}\\
&\qquad \geq \left(\frac{\lambda_s}{2}+\frac{\mu_s}{n}\right)\int_\O(\dive u)^2 \, dx+ \liminf_{\e\to 0}\int_\O (1-\chi_\e) \tilde f(e_D(u_\e))+\chi_\e \tilde g_\e(e_D(u_\e))\, dx.
\end{align*}
For a further use, we also have that the sequence $v_\e:=\sqrt{\e} u_\e$ is bounded in $H^1(\O;\R^n)$, so that $v_\e \wto 0$ weakly in $H^1(\O;\R^n)$ and $\dive v_\e \to 0$ strongly in $L^2(\O)$.
    
\medskip
    
\noindent {\it Step 2a: The two-dimensional case.}  Since every matrix  $\xi \in \mathbb M^{2 \times 2}_D$ satisfies $\det(\xi)\leq0$, Lemma~\ref{lem:aodotb} ensures that $\xi=a \odot b$ for some $a$ and $b \in \R^2$. Therefore, according to Young's inequality,
$$\tilde g_\e(e_D(u_\e)) \geq 2 \sqrt{\kappa \mu_w}|e_D(u_\e)| = \sqrt{2\kappa \tilde h(e_D(u_\e))}.$$
Hence, since $\widetilde W=\tilde f \, \Box \, \sqrt{2\kappa \tilde h}$, 
$$\liminf_{\e\to 0}\widetilde E_\e(u_\e,\chi_\e)\geq \left(\frac{\lambda_s}{2}+\frac{\mu_s}{2}\right)\int_\O(\dive u)^2\, dx+
\liminf_{\e\to 0}\int_{\O}\widetilde W(e_D(u_\e))\,dx$$
and we conclude by the weak* lower semicontinuity theorem for convex functionals of measures.
    
\medskip
    
\noindent {\it Step 2b: The three-dimensional case.}
We use the same notation and the same arguments as for the three-dimensional case in Theorem~\ref{thm:alpha=1}. We first note that since $f=\tilde f$ and $g_\e=\tilde g_\e$ on $\mathbb M^{3 \times 3}_D$, for all open sets $\omega \subset \O$, all $\varphi \in \mathcal C_c(\omega)$ with $0 \leq \varphi \leq 1$, and all $A \in {\rm conv}(M)$, we have for every $\gamma > 0$ and $\e > 0$ small enough (see the corresponding argument in the Hencky case),
 \begin{align*} 
&\int_\omega (1-\chi_\e) \tilde f(e_D(u_\e))+\chi_\e  \tilde g_\e(e_D(u_\e)) \, dx \\
&\qquad\geq  \int_\omega \varphi \, \big[ (1-\chi_\e) f(e_D(u_\e)) + \chi_\e g_\e (e_D(u_\e)) \big] \, dx\\
&\qquad\geq (1-\gamma) \int_\omega \varphi \,  (f \,\Box\,  \sqrt{2\kappa h_A})(e_D(u_\e)) \, dx -2\e\mu_w  \int_\omega \varphi\, A:\cof(e_D(u_\e))\, dx.
 \end{align*}
We claim that
\begin{equation}
 \label{eq:cof}
\cof(e_D(v_\e)) - \cof(e(v_\e))\to 0 \quad \text{ strongly in }L^1(\O;\Mst). 
\end{equation}
Indeed, any matrix  $\xi \in \Mst$ can be written as $\xi=P\Lambda P^{-1}$ with $\Lambda={\rm diag}(\xi_1,\xi_2,\xi_3) \in \Mst$ diagonal and $P \in SO(3)$. Then, $\xi_D=P\Lambda_DP^{-1}$ with 
$$\Lambda_D={\rm diag}\left(\xi_1-\frac{\tr\xi}{3},\xi_2-\frac{\tr\xi}{3},\xi_3-\frac{\tr\xi}{3}\right),$$
and Lemma~\ref{lem:polyconvex} shows that $\cof(\xi)=P\cof(\Lambda)P^{-1}$ and $\cof(\xi_D)=P\cof(\Lambda_D)P^{-1}$ with $\cof(\Lambda)={\rm diag}(\xi_2\xi_3,\xi_1\xi_3,\xi_1\xi_2)$ and
$$\cof(\Lambda_D)={\rm diag}\left(\Big(\xi_2-\frac{\tr\xi}{3}\Big)\Big(\xi_3-\frac{\tr\xi}{3}\Big),\Big(\xi_1-\frac{\tr\xi}{3}\Big)\Big(\xi_3-\frac{\tr\xi}{3}\Big),\Big(\xi_1-\frac{\tr\xi}{3}\Big)\Big(\xi_2-\frac{\tr\xi}{3}\Big)\right).$$
Therefore, $\cof(\xi)-\cof(\xi_D)=P\big(\cof(\Lambda)-\cof(\Lambda_D)\big)P^{-1}$
with
$$\cof(\Lambda)-\cof(\Lambda_D)=-\frac{(\tr\xi)^2}{9}{\rm Id} +\frac{\tr\xi}{3}\,{\rm diag}(\xi_2+\xi_3,\xi_1+\xi_3,\xi_1+\xi_2)$$
Specifying the previous expression to $\xi=e(v_\e)$ and observing that the eigenvalues of $e(v_\e)$ are bounded in $L^2(\O)$ uniformly in $\e>0$ (since the spectral radius satisfies $\rho(e(v_\e))\leq |e(v_\e)|$), while $\dive v_\e \to0$ strongly in $L^2(\O)$, we finally deduce~\eqref{eq:cof}.

Arguing as in the proof of Theorem~\ref{thm:alpha=1}, we conclude that 
$$
 \limsup_{\e \to 0} \int_{\omega}\varphi\,A: \cof(e_D(v_{\e}))\, dx= \limsup_{\e \to 0} \int_{\omega}\varphi\,A: \cof(e(v_{\e}))\, dx\leq 0.
$$
 Moreover, by the weak* lower semicontinuity theorem for convex functionals of measures, we have
 \begin{align*}
 &\liminf_{\e \to 0} \int_\omega \varphi \, (f \wedge \sqrt{2\kappa h_A})(e_D(u_\e))\, dx \\
 &\qquad\geq  \liminf_{\e \to 0}\int_\omega \varphi \, (f\,\Box\, \sqrt{2\kappa h_A})(e_D(u_\e))\, dx\\
&\qquad\geq \int_\omega \varphi \, (f\,\Box\, \sqrt{2\kappa h_A})(e_D(u))\, dx + \int_\omega \varphi(f\,\Box\, \sqrt{2\kappa h_A})^\infty\left(\frac{dE^s u}{d|E^s u|} \right)d|E^su|.
 \end{align*}

The remainder of the proof follows the lines of Theorem~\ref{thm:alpha=1}. Passing to the supremum over $\varphi\in \mathcal C_c(\omega)$, $0\leq\varphi\leq1$, and over $A\in {\text{conv}}(M)$, we find in a similar fashion as before, in particular letting $\gamma \to 0$, that
 \begin{align*}
 \liminf_{\e \to 0} \widetilde E_\e(u_\e,\chi_\e) 
 &\geq \left(\frac{\lambda_s}{2}+\frac{\mu_s}{3}\right) \int_\O(\dive u)^2\, dx+ \int_\O  (f\,\Box\, \sqrt{2\kappa h})(e_D(u))\, dx\\
 &\hspace{2cm} +\int_\O(f\,\Box\, \sqrt{2\kappa h})^\infty\left(\frac{dE^s u}{d|E^s u|} \right)d|E^su|\\
&\geq \left(\frac{\lambda_s}{2}+\frac{\mu_s}{3}\right) \int_\O(\dive u)^2\, dx+ \int_\O  ( \tilde f\,\Box\, \sqrt{2\kappa \tilde h})(e_D(u))\, dx\\
& \hspace{2cm}+ \int_\O(\tilde f\,\Box\, \sqrt{2\kappa \tilde h})^\infty\left(\frac{dE^s u}{d|E^s u|} \right)d|E^su|\\
&= \left(\frac{\lambda_s}{2}+\frac{\mu_s}{3}\right) \int_\O(\dive u)^2 \, dx+ \int_\O \widetilde W(e_D(u))\, dx+\int_\O\widetilde W^\infty\left(\frac{dE^s u}{d|E^s u|} \right)d|E^su|.
\end{align*}
Note that in the first inequality the inf-convolutions are to be understood in the full space $\mathbb M^{3 \times 3}_{\rm sym}$, while in the second inequality the inf-convolutions are to be understood in $\mathbb M^{3 \times 3}_D$. Moreover, we have used $f(\xi) \geq f(\xi_D)=\tilde f(\xi_D)$ and $\sqrt{2\kappa h(\xi)} \geq \sqrt{2\kappa h(\xi_D)}=\sqrt{2\kappa \tilde h(\xi_D)}$. The proof of the theorem is complete.
\end{proof}

The following proposition is the corresponding of Proposition~\ref{prop:Walpha_char} in the Tresca regime.	

\begin{prop} \label{prop:Walpha_div}
We have
\begin{multline*}
\widetilde W =\sup\Big\{\varphi : \Msd \to \R \text{ convex, }\varphi(\xi) \leq \tilde f(\xi) \text{ for all }\xi \in \Msd\\
\varphi(a \odot b) \leq \sqrt{2\kappa \mathbf{A}_w(a\odot b):(a \odot b)} \text{ for all }a,b \in \R^n \text{ with } a \cdot b = 0 \Big\}.
\end{multline*}
\end{prop}
	
\begin{proof}
The proof is very similar to that of Proposition~\ref{prop:Walpha_char}, hence we only sketch it. We only need to check that
the function $\widetilde H:\Msd \to [0,+\infty]$ defined by
$$\widetilde H(\xi):=
\begin{cases}
2\sqrt{\kappa \mu_w}|\xi| & \text{ if }\xi=a\odot b\in\Msd \text{ for some }a,b \in \R^n \text{ with } a \cdot b = 0,\\
+\infty & \text{ otherwise,}
\end{cases}$$
satisfies
$\widetilde H^*(\tau)=+\infty$ for all $\tau \notin \widetilde K$. Let us fix $\tau \notin \widetilde K$, i.e.\ 
\begin{equation}\label{tG1}
\widetilde G(\tau)=\frac{(\tau_1-\tau_n)^2}{4\mu_w}>2\kappa. 
\end{equation}
It is not restrictive to assume that $\tau$ is diagonal with ordered eigenvalues $\tau_1\leq\dots\leq\tau_n$.
We denote by $R_D$ the set of the diagonal rank-one symmetric deviatoric matrices with ordered eigenvalues $\xi_1=-\xi_n\leq0 = \cdots = 0\leq\xi_n$. Then, by definition,
\begin{equation}\label{tR1}
\widetilde H^*(\tau) \geq\sup_{\xi\in R_D} \bigl\{ \tau:\xi - 2\sqrt{\kappa \mu_w}|\xi|\bigr\}.
\end{equation} 
Setting
$\tau_d:=\tau_n-\tau_1$, equations~\eqref{tG1} and~\eqref{tR1} become
\[\label{tG2}
2\kappa<\widetilde G(\tau)=\frac{\tau_d^2}{4\mu_w},	
\] 
\[\label{tR2}
\widetilde H^*(\tau) \geq\sup_{\xi_d\geq0} \left\{ \frac{\tau_d\xi_d}{2} - \sqrt{2 \kappa\mu_w}\xi_d\right\}=+\infty.
\]				 
This concludes the proof.	
\end{proof}

\section*{Acknowledgements}
F.I. has been a recipient of scholarships from the Fondation Sciences Math\'ematiques de Paris, Emergence Sorbonne Universit\'es, and the S\'ephora-Berrebi Foundation, and gratefully acknowledges their support. This project has received funding from the European Research Council (ERC) under the European Union's Horizon 2020 research and innovation programme, grant agreement No 757254 (SINGULARITY).



\end{document}